\documentclass[12pt]{amsart}
\usepackage{amssymb,amsmath}
\usepackage{geometry}    
\usepackage{mathrsfs}

\usepackage{graphicx}

\usepackage{vmargin}
\setmargrb{1in}{1in}{1in}{1in}

\newtheorem{theorem}{Theorem}[section]

\newtheorem{proposition}[theorem]{Proposition}
\newtheorem{corollary}[theorem]{Corollary}
\theoremstyle{definition}

\newtheorem{remark}{Remark}

\newtheorem{problem}{Problem}

\begin{document}

\title[An equivalence principle]{An equivalence principle between polynomial and simultaneous Diophantine approximation}

\author{Johannes Schleischitz} 

\thanks{Supported by the Schr\"odinger Scholarship J 3824 of the Austrian Science Fund (FWF).\\
	University of Ottawa, Department of Mathematics and Statistics, King Edward 585, ON K1N 6N5 \\
	johannes.schleischitz@univie.ac.at}


\begin{abstract}
We show that Mahler's classification of real numbers $\zeta$
with respect to 
the growth of the sequence $(w_{n}(\zeta))_{n\geq 1}$ is equivalently induced by certain natural assumptions
on the decay of the sequence $(\lambda_{n}(\zeta))_{n\geq 1}$
concerning simultaneous rational approximation. 
Thereby we obtain a much clearer picture on simultaneous
approximation to successive powers of a real number in general.
Another variant of the Mahler classification
concerning uniform approximation by algebraic
numbers is derived as well.
Our method has several applications to classic exponents
of Diophantine approximation and metric theory. We deduce 
estimates on the Hausdorff dimension
of well-approximable vectors on the Veronese curve and
refine the best known upper bound for the exponent
$\widehat{\lambda}_{n}(\zeta)$ for even $n\geq 4$.
\end{abstract}

\maketitle

{\footnotesize{

Math Subject Classification 2010: 11J13, 11J82, 11J83}}

\vspace{4mm}

\section{Classical exponents of Diophantine approximation} \label{intro}

In this paper we establish a link between 
classical intensely studied Diophantine approximation problems. 
Let $\zeta$ be a transcendental real number and $m,n$ be 
positive integers.
On one hand, we are concerned with
small polynomial evaluations $\vert P(\zeta)\vert$ for integer polynomials
$P$ of degree at most $n$, in terms of the height of $P$. 
This problem is known to be closely connected to approximation
to $\zeta$ by real algebraic numbers of degree 
at most $n$.
On the other hand, we deal with simultaneous rational approximation to 
$(\zeta,\zeta^2,\ldots,\zeta^m)$.
The latter problem is again directly linked to approximation
to $\zeta$ by real algebraic numbers (resp. integers) of degree 
at most $m$ (resp. $m+1$), see Davenport and Schmidt~\cite{davsh}. 
We establish connections between these classical problems, for
suitable pairs $m,n$. This will lead to a better understanding
of both classical problems individually.

Mahler introduced the classical exponent $w_{n}(\zeta)$
as the supremum of real numbers $w$ such that
\begin{equation}  \label{eq:www}
0<\vert P(\zeta)\vert \leq H(P)^{-w},  
\end{equation}
has infinitely many solutions $P\in{\mathbb{Z}[T]}$ of degree at most $n$.
Here $H(P)$ is the maximum modulus of the coefficients of $P$. 
Bugeaud and Laurent defined the uniform 
exponent $\widehat{w}_{n}(\zeta)$ 
as the supremum of $w\in{\mathbb{R}}$ such that the system 
\begin{equation}  \label{eq:w}
H(P) \leq X, \qquad  0<\vert P(\zeta)\vert \leq X^{-w},  
\end{equation}
has a solution $P\in{\mathbb{Z}[T]}$ of degree at most $n$
for all large $X$. It is easy to see the definition
of $w_{n}(\zeta)$ is equivalent to requiring \eqref{eq:w} to be satisfied
for certain arbitrarily large $X$. 
By Dirichlet's Theorem we have
\begin{equation} \label{eq:wmono}
w_{n}(\zeta)\geq \widehat{w}_{n}(\zeta)\geq n, \qquad\qquad n\geq 1.
\end{equation}
Moreover, since we admit more polynomials as $n$ increases
we obviously have
\begin{equation}  \label{eq:obv}
w_{1}(\zeta)\leq w_{2}(\zeta)\leq \cdots, \qquad \widehat{w}_{1}(\zeta)\leq \widehat{w}_{2}(\zeta)\leq \cdots.
\end{equation}
In this paper we focus on the best approximation 
exponents $w_{n}$, however some contributions to 
the uniform exponents $\widehat{w}_{n}$ will arise
as a byproduct as well,
particularly in Section~\ref{widhat}.
It is conjectured that \eqref{eq:wmono} and \eqref{eq:obv}
are the only limitations on sequences $(w_{n}(\zeta))_{n\geq 1}$
if we may choose any transcendental real $\zeta$.
We recall this partial assertion of the {\em Main problem} in~\cite[Section~3.4, page~61]{bugbuch}
on the {\em joint spectrum of $(w_{n}(\zeta))_{n\geq 1}$}.

\begin{problem} \label{mp}
Let $(w_{n})_{n\geq 1}$ be a non-decreasing sequence of real numbers with $w_{n}\geq n$.
Does there exist $\zeta$ such that $w_{n}(\zeta)=w_{n}$ 
simultaneously for all $n\geq 1$?
\end{problem}

Although a positive answer is strongly expected, only special cases have
been verified.
Mahler classified the transcendental real numbers in terms of the
growth of the sequence $(w_{n}(\zeta))_{n\geq 1}$. He called           
a transcendental real number $\zeta$ a $U_{m}$-number 
if $w_{m}(\zeta)=\infty$ and $m$ is the smallest such index.
The set of $U$-numbers is defined as the disjoint union of the sets of $U_{m}$-numbers over $m\geq 1$. Due to Mahler, 
a number $\zeta$ is called a $T$-number if $w_{n}(\zeta)<\infty$ for 
all $n\geq 1$, but $\limsup_{n\to\infty} w_{n}(\zeta)/n=\infty$ holds.
Finally, the remaining numbers for which $w_{n}(\zeta)/n\ll 1$ are called $S$-numbers. A famous result of Sprind\v{z}uk~\cite{sprindzuk} states that 
almost all real numbers in the sense of Lebesgue measure 
satisfy $w_{n}(\zeta)=n$ for all $n\geq 1$,
in particular they are $S$-numbers. Building up on results
of Baker and Schmidt~\cite{baks},
Bernik~\cite{bernik} refined this in a metrical sense 
by showing the formula
\begin{equation} \label{eq:bern}
\dim(\{\zeta\in\mathbb{R}: w_{n}(\zeta)\geq w\})=\dim(\{\zeta\in\mathbb{R}: w_{n}(\zeta)= w\})= \frac{n+1}{w+1}, \qquad \qquad w\geq n.
\end{equation}
Here and in the sequel $\dim$ denotes the Hausdorff dimension,
see~\cite{falconer} for an introduction. 
Generalizing~\cite{sprindzuk},~\cite{bernik} to
non-degenerate manifolds, and in other subtle ways,
is an active topic in modern metric
Diophantine approximation. However, this is not a major concern
of this paper and we only refer to~\cite{huang} for a recent, general result dealing with planar curves.
By \eqref{eq:bern} the sets of $U$-numbers and $T$-numbers in fact both have
dimension $0$. However, they are well-known to be non-empty, see
LeVeque~\cite{leveque} and Schmidt~\cite{tnumbers}. 
We refer to~\cite{bugbuch} for 
more results connected to Mahler's classification and related topics.

We want to relate the exponents $w_{n}(\zeta)$ to the
exponents of simultaneous approximation introduced by 
Bugeaud and Laurent~\cite{buglau}.
They define the exponent $\lambda_{n}(\zeta)$ as
the supremum of $\lambda\in{\mathbb{R}}$ such that the system 
\begin{equation}  \label{eq:lambda}
1\leq \vert x\vert \leq X, \qquad \max_{1\leq i\leq n} \vert \zeta^{i}x-y_{i}\vert \leq X^{-\lambda},  
\end{equation}
has a solution $(x,y_{1},y_{2},\ldots, y_{n})\in{\mathbb{Z}^{n+1}}$ 
for arbitrarily large values of $X$. 
Similarly, they denote by $\widehat{\lambda}_{n}(\zeta)$
the supremum of $\lambda$ such that \eqref{eq:lambda} has a 
solution for all $X\geq X_{0}$. For any transcendental real $\zeta$ 
and $n\geq 1$, by Dirichlet's Theorem these exponents satisfy
\begin{equation} \label{eq:ldiri}
\lambda_{n}(\zeta)\geq \widehat{\lambda}_{n}(\zeta)\geq \frac{1}{n}.
\end{equation}
Moreover from the definition we see that
\begin{equation} \label{eq:reihenfolge}
\lambda_{1}(\zeta)\geq \lambda_{2}(\zeta)\geq \cdots, 
\qquad \widehat{\lambda}_{1}(\zeta)\geq \widehat{\lambda}_{2}(\zeta)\geq \cdots.
\end{equation}
Khintchine~\cite{khint} was the first to show a relation between
polynomial and simultaneous approximation. His transference principle 
asserts
\begin{equation} \label{eq:khintchine}
\frac{w_{n}(\zeta)}{(n-1)w_{n}(\zeta)+n}\leq \lambda_{n}(\zeta)\leq \frac{w_{n}(\zeta)-n+1}{n}.
\end{equation}
So far, essentially the transference principle has been the 
only tool for comparing the sequences
$(w_{n}(\zeta))_{\geq 1}$ and $(\lambda_{n}(\zeta))_{n\geq 1}$. It provides very limited information. 
In particular it is far from clear what to expect for the joint behavior of the sequences $(\lambda_{n}(\zeta))_{n\geq 1}$,
in contrast to the precise conjecture in Problem~\ref{mp}.
In Section~\ref{spectrum} below we will quote the few previously known 
results.
Our main result Theorem~\ref{nabum} will provide vastly refined information on the interaction of the
two sequences $(w_{n}(\zeta))_{n\geq 1}$ and 
$(\lambda_{n}(\zeta))_{n\geq 1}$ for given $\zeta$, and thereby lead
to a much better understanding of the latter sequence as a byproduct.
Roughly speaking our method is based on comparison of $w_{n}(\zeta)$
with $\lambda_{m}(\zeta)$ for suitable pairs $m,n$. In fact we
usually take $m$ to be reasonably larger than $n$, in contrast to
$m=n$ in \eqref{eq:khintchine}. We emphasize that
for our method it is crucial that we deal with successive powers
of a real number, whereas \eqref{eq:khintchine} remains valid
for any vector $\underline{\zeta}\in\mathbb{R}^{n}$
that is $\mathbb{Q}$-linearly independent together with $\{1\}$, 
with accordingly altered definitions of the exponents.

\section{Equivalence principles} \label{eqprin}

\subsection{Simultaneous approximation}
In our main result we link the sequences of exponents $(w_{n}(\zeta))_{n\geq 1}$ and $(\lambda_{n}(\zeta))_{n\geq 1}$.

\begin{theorem}[Equivalence principle I] \label{nabum}
Let $\zeta$ be a transcendental real number. Then $\zeta$ is a $U$-number if and only if
\begin{equation} \label{eq:uu}
\lim_{n\to\infty} \lambda_{n}(\zeta)>0.
\end{equation}
More precisely, if $\zeta$ is a $U_{m}$-number, then $\lambda_{n}(\zeta)=\frac{1}{m-1}$ for all sufficiently large $n$.
Moreover, $\zeta$ is a $T$-number if and only if
\[
\lim_{n\to\infty} \lambda_{n}(\zeta)=0, \qquad \limsup_{n\to\infty} n\lambda_{n}(\zeta)=\infty. 
\]
Finally $\zeta$ is an $S$-number if and only if
\[
\limsup_{n\to\infty} n\lambda_{n}(\zeta)<\infty.
\]
\end{theorem}

The theorem shows that Mahler's classification can be equivalently obtained by natural assumptions
on the decay of the sequence $(\lambda_{n}(\zeta))_{n\geq 1}$.
The transference principle \eqref{eq:khintchine} admits the conclusion
$\lim_{n\to\infty} \lambda_{n}(\zeta)=0$ upon 
$\liminf_{n\to\infty} w_{n}(\zeta)/n=1$, a reasonably stronger condition
than $\zeta$ being no $U$-number. Hence \eqref{eq:khintchine} does not
rule out \eqref{eq:uu} even for certain $S$-numbers.
As a first corollary we determine all limits of the 
sequences $(\lambda_{n}(\zeta))_{n\geq 1}$.

\begin{corollary}  \label{lamkor}
The set $\mathcal{S}$ of all values $\lim_{n\to\infty} \lambda_{n}(\zeta)$ as $\zeta$ attains any transcendental real number
is precisely the countable set $\mathcal{S}=\{0,\infty\}\cup \{1,\frac{1}{2},\frac{1}{3},\ldots \}$. 
\end{corollary}

\begin{proof}
For any $S$-number and $T$-number the limit is $0$ by Theorem~\ref{nabum}. For a $U_{m}$-number the limit 
is $1/(m-1)$ again by Theorem~\ref{nabum}. The claim follows.
\end{proof}

\begin{remark}
Previous results recalled
in Section~\ref{spectrum} below could have settled
$\{0,1,\infty\} \subseteq \mathcal{S}\subseteq [0,1]\cup\{\infty\}$. 
Indeed, the inclusion
$\{0,1,\infty\} \subseteq \mathcal{S}$ follows
from Sprind\v{z}uk~\cite{sprindzuk} 
and Bugeaud~\cite[Theorem~4, Corollary~2]{bug},
whereas the 
consequence~\cite[Corollary~1.9]{schlei} 
of \eqref{eq:trampel} implies $(1,\infty)\cap \mathcal{S}= \emptyset$. 
\end{remark}

We can provide
effective relations between the sequences $(w_{n}(\zeta))_{\geq 1}$ 
and $(\lambda_{n}(\zeta))_{n\geq 1}$.
We recall the notion of                                                                                  
the order $\tau(\zeta)$ of a $T$-number~\cite{bugbuch}, defined as
\[
\tau(\zeta)=\limsup_{n\to\infty} \frac{\log w_{n}(\zeta)}{\log n}.
\]
We have $\tau(\zeta)\in[1,\infty]$ for any $T$-number $\zeta$ by \eqref{eq:wmono}.
All $T$-numbers that have been constructed so far have 
order $\tau(\zeta)\geq 3$, 
and R. Baker~\cite{rbaker}
conversely constructed $T$-numbers of the given degree $\tau(\zeta)\in[3,\infty]$.
See also~\cite[Theorem~7.2]{bugbuch}, however there seems to be a problem in the proof 
as in (7.28) a stronger estimate than the assumption (7.24) is used.
A positive answer to
Problem~\ref{mp} would clearly imply that $T$-numbers of 
any degree $\tau(\zeta)\in[1,\infty]$ exist. 
We propose a somehow dual order $\sigma(\zeta)$, defined as 
\[
\sigma(\zeta)=\limsup_{n\to\infty} \frac{\log \lambda_{n}(\zeta)}{\log n}.
\]
It follows from \eqref{eq:ldiri} and \eqref{eq:reihenfolge} that
$\sigma(\zeta)\in[-1,0]$ for any $\zeta$ which is not a Liouville number
(i..e a $U_{1}$-number). In fact even 
$\log \lambda_{n}(\zeta)/\log n\leq 0$ for all large $n$ 
by~\cite[Theorem~1.6]{schlei}.
Further define
\begin{equation} \label{eq:defwoben}
\overline{w}(\zeta):= \limsup_{n\to\infty} \frac{w_{n}(\zeta)}{n}, \qquad \overline{\lambda}(\zeta):= \limsup_{n\to\infty} n\lambda_{n}(\zeta), 
\end{equation}
and
\begin{equation} \label{eq:defwunten}
\underline{w}(\zeta):= \liminf_{n\to\infty} \frac{w_{n}(\zeta)}{n}, \qquad \underline{\lambda}(\zeta):= \liminf_{n\to\infty} n\lambda_{n}(\zeta). 
\end{equation}
The set of $S$-numbers equals the set
of numbers with $\overline{w}(\zeta)<\infty$.
For $S$-numbers and $T$-numbers of order $\tau(\zeta)=1$, 
the quantities $\overline{w}(\zeta), \underline{w}(\zeta)$ provide
a refined measure. Similarly 
$\overline{\lambda}(\zeta), \underline{\lambda}(\zeta)$ 
refine $\sigma(\zeta)$.
We obtain connections between the quantities 
as follows.

\begin{theorem} \label{genau}
For any real transcendental $\zeta$ we have
\begin{equation} \label{eq:leftie} 
\frac{(\overline{w}(\zeta)+1)^{2}}{4\overline{w}(\zeta)} \leq
\overline{\lambda}(\zeta) \leq \overline{w}(\zeta)+2, \qquad  \frac{(\underline{w}(\zeta)+1)^{2}}{4\underline{w}(\zeta)} \leq 
\underline{\lambda}(\zeta)\leq \underline{w}(\zeta)+2,   
\end{equation}
and moreover
\begin{equation} \label{eq:hadschi}
\sigma(\zeta)=-\frac{1}{\tau(\zeta)}.
\end{equation}
\end{theorem}

In the theorem and generally for the sequel
we always agree on $1/\infty=0$ and $1/0=+\infty$.
There is no reason to believe that the bounds in \eqref{eq:leftie} are optimal. 
It is tempting to conjecture that $\overline{w}(\zeta)= \overline{\lambda}(\zeta)$ and $\underline{w}(\zeta)= \underline{\lambda}(\zeta)$
hold for any transcendental real $\zeta$.

\subsection{Uniform approximation by algebraic numbers}

In this section we establish another equivalence principle.
We connect the Mahler classification with exponents of 
uniform approximation
to a real number by algebraic numbers of degree bounded by some $n$.
Let $w_{n}^{\ast}(\zeta)$ and $\widehat{w}_{n}^{\ast}(\zeta)$
be the supremum of $w^{\ast}$ such that the system
\begin{equation} \label{eq:stern}
H(\alpha)\leq X, \qquad
\vert \zeta-\alpha\vert \leq H(\alpha)^{-1}X^{-w^{\ast}}
\end{equation}
has a real algebraic solution $\alpha$ of degree at most $n$, for arbitrarily large and all large $X$, respectively.
Here $H(\alpha)=H(P)$ for $P$ the (up to sign) unique minimal polynomial of $\alpha$ with coprime integral coefficients. 
These exponents are closely linked 
to the polynomial exponents $w_{n}(\zeta), \widehat{w}_{n}(\zeta)$. 
In particular, 
the same partition of the transcendental real numbers is induced by replacing $w_{n}$ in the Mahler classification above 
by $w_{n}^{\ast}$, 
as proposed by Koksma. Indeed this is an immediate consequence of
the estimates
\begin{equation} \label{eq:duzu}
w_{n}^{\ast}(\zeta)\leq w_{n}(\zeta)\leq w_{n}^{\ast}(\zeta)+n-1,
\end{equation}
from ~\cite[Lemma~A8]{bugbuch}. The analogous estimates hold
for the uniform exponents, and together with upper bounds 
by Davenport and Schmidt~\cite{davsh} we may comprise
\begin{equation} \label{eq:bound2}
\widehat{w}_{n}^{\ast}(\zeta)\leq \widehat{w}_{n}(\zeta)\leq
\min\{ 2n-1, \widehat{w}_{n}^{\ast}(\zeta)+n-1 \}.
\end{equation}
The bound $2n-1$ has in fact been slightly improved
in~\cite{bschlei} and further recently in~\cite{ichacta}.
We show that Mahler's classification is obtained as well by 
imposing natural 
assumptions on the sequence of uniform 
exponents $\widehat{w}_{n}^{\ast}(\zeta)$.

\begin{theorem}[Equivalence principle II] \label{dadkor}
Let $\zeta$ be a transcendental real number. Then $\zeta$ is a $U$-number
if and only if 
\begin{equation} \label{eq:limites}
\lim_{n\to\infty} \widehat{w}_{n}^{\ast}(\zeta)<\infty.
\end{equation}
More precisely, if $\zeta$ is a $U_{m}$-number, then $\widehat{w}_{n}^{\ast}(\zeta)\in[m-1,m]$ for all sufficiently large $n$. 
Moreover, $\zeta$ is a $T$-number if and only if 
\begin{equation} \label{eq:tnu}
\lim_{n\to\infty} \widehat{w}_{n}^{\ast}(\zeta)=\infty, \qquad
\liminf_{n\to\infty} \frac{\widehat{w}_{n}^{\ast}(\zeta)}{n}=0.
\end{equation}
Finally $\zeta$ is an $S$-number if and only if 
there exists a constant $\delta>0$ such that
$\widehat{w}_{n}^{\ast}(\zeta)\geq \delta n$ for all $n\geq 1$.
\end{theorem}

\begin{remark} \label{reeh}
Several variants of equivalence principle II can be derived similarly.
For example one can fix the degree of the algebraic numbers
in \eqref{eq:stern} equal to $n$,
or restrict to approximation by algebraic integers or algebraic units.
See for example~\cite{teu},~\cite{davsh}, or~\cite{ichfix}.
\end{remark}

Define the quantities
\[
\underline{\widehat{w}}^{\ast}(\zeta) = 
\liminf_{n\to\infty} \frac{\widehat{w}_{n}^{\ast}(\zeta)}{n}, \qquad\qquad \overline{\widehat{w}}^{\ast}(\zeta) = 
\limsup_{n\to\infty} \frac{\widehat{w}_{n}^{\ast}(\zeta)}{n},
\]
and further let
\[
\theta(\zeta)=\liminf_{n\to\infty} 
\frac{\log \widehat{w}_{n}^{\ast}(\zeta)}{\log n}.
\]
By \eqref{eq:bound2}
we have $0\leq \underline{\widehat{w}}^{\ast}(\zeta)\leq
\overline{\widehat{w}}^{\ast}(\zeta)\leq 2$ and $\theta(\zeta)\in[0,1]$.
An effective version of the second equivalence principle reads as follows.

\begin{theorem} \label{difizil}
Let $\zeta$ be any transcendental real number. We have
\begin{equation} \label{eq:hum}
\frac{1}{\overline{w}(\zeta)+2}\leq 
\underline{\widehat{w}}^{\ast}(\zeta)
\leq \min\left\{ \underline{w}(\zeta),
\frac{4}{\overline{w}(\zeta)}\right\},
\end{equation}
and
\begin{equation} \label{eq:bug}
\frac{1}{\underline{w}(\zeta)+2}\leq \overline{\widehat{w}}^{\ast}(\zeta)
\leq \min\left\{ \overline{w}(\zeta),
\frac{4}{\underline{w}(\zeta)}\right\}.
\end{equation}
Moreover
\begin{equation} \label{eq:hadschibratschi}
\theta(\zeta)=\frac{1}{\tau(\zeta)}=-\sigma(\zeta).
\end{equation}
\end{theorem} 

Apparently for large
values of $\overline{w}(\zeta)$ and $\underline{w}(\zeta)$, 
the respective lower and upper bound differ roughly by the same
factor $4$ as in Theorem~\ref{genau}.
This is surprising as the proofs of upper bounds in 
Theorem~\ref{difizil} is unrelated to the proof of
Theorem~\ref{genau}. 
It is hard to predict if this factor $4$ has any deeper meaning.
Note that for the similarly defined 
quantities $\underline{w}^{\ast}(\zeta), \overline{w}^{\ast}(\zeta)$,
Wirsing's~\cite{wirsing}
estimate $w_{n}^{\ast}(\zeta)\geq (w_{n}(\zeta)+1)/2\geq (n+1)/2$
and \eqref{eq:duzu} imply
\[
\frac{1}{2}\leq \max\left\{\frac{\overline{w}(\zeta)}{2},\overline{w}(\zeta)-1\right\}
\leq 
\overline{w}^{\ast}(\zeta)\leq \overline{w}(\zeta),
\]
and
\[
\frac{1}{2}\leq \max\left\{\frac{\underline{w}(\zeta)}{2},\underline{w}(\zeta)-1\right\}\leq \underline{w}^{\ast}(\zeta)\leq \underline{w}(\zeta).
\]

Thus
$\tau(\zeta)$ equals the order $\tau^{\ast}(\zeta)$
obtained by replacing $w_{n}(\zeta)$ by $w_{n}^{\ast}(\zeta)$.
Hence a variant of Theorem~\ref{difizil}, in terms of 
quantities derived from
$w_{n}^{\ast}(\zeta)$ and $\widehat{w}_{n}^{\ast}(\zeta)$ only,
can be formulated. We do not explicitly state it.
 
Similar to Corollary~\ref{lamkor}, we can ask for the set 
$\mathscr{W}$ of limits 
of the sequences $(\widehat{w}_{n}^{\ast}(\zeta))_{n\geq 1}$
as $\zeta$ attains every real number.
We conjecture that $\mathscr{W}=\{ \infty\}\cup \{1,2,3,\ldots\}$. 
However, Theorem~\ref{dadkor} only admits the inclusion 
$\mathscr{W}\supseteq \{1,\infty\}$,
and conversely we cannot even exclude $\mathscr{W}=[1,\infty]$.

\subsection{Comments and outline of the following sections}
\label{brief}
We recapitulate that in Section~\ref{eqprin} we derived four equivalent
definitions of the Mahler classification in terms of the
sequences $(w_{n}(\zeta))_{n\geq 1}, (\lambda_{n}(\zeta))_{n\geq 1},
(w_{n}^{\ast}(\zeta))_{n\geq 1}$ and
$(\widehat{w}_{n}^{\ast}(\zeta))_{n\geq 1}$ respectively. It is natural
to ask if the sequences $(\widehat{w}_{n}(\zeta))_{n\geq 1}$
and $(\widehat{\lambda}_{n}(\zeta))_{n\geq 1}$ can be somehow
included in the picture. However, almost all $S$-numbers 
satisfy $\widehat{w}_{n}(\zeta)=n$
and $\widehat{\lambda}_{n}(\zeta)=1/n$ for all $n\geq 1$
by Sprind\v{z}uk~\cite{sprindzuk}, and any Liouville number
(i.e. a $U_{1}$-number) shares the same property 
by~\cite[Corollary~5.2]{schlei}.
Hence it seems not to be possible. We also want to 
point out that although Theorem~\ref{nabum} and 
Corollary~\ref{lamkor}
provide much new information on the exponents $\lambda_{n}(\zeta)$,
they are insufficient when it comes to addressing certain more subtle questions on
the decay of the sequences $(\lambda_{n}(\zeta))_{n\geq 1}$
within the interval $(0,1)$. For 
example~\cite[Problem~1.11]{schlei} remains open,
asking if the estimate 
\[
\lambda_{m}(\zeta)\geq \frac{n\lambda_{n}(\zeta)-m+n}{m}
\]
holds for any integers $m\geq n\geq 1$ and any
real number $\zeta$. The answer is affirmative when $n$ divides $m$,
see~\cite[Lemma~1]{bug}, or when $\lambda_{m}(\zeta)>1$ even 
with equality~\cite[Corollary~1.10]{schlei}. See also
Section~\ref{spectrum} below.

We give a brief outline of the upcoming sections.
In Section~\ref{refin} below we establish several
partial results, which combine to Theorem~\ref{nabum} and 
Theorem~\ref{genau}. 
These partial results have additional
interesting consequences on their own,
gathered in Section~\ref{metric}. There
we refine the upper
bound for the uniform exponents $\widehat{\lambda}_{n}(\zeta)$ for
even $n$. Furthermore we study the consequences of the equivalence principle to 
the metric problem of determining the Hausdorff dimension of vectors
on the Veronese curve that are simultaneously 
approximable to a given order. Moreover,
for numbers $\zeta$ that admit many very small 
evaluations at integer polynomials of bounded degree,
we provide a rate of decay for 
the exponents $\lambda_{n}(\zeta)$ for large $n$. Suitable
numbers include the Champerowne number and any number with
the property $\widehat{w}_{n}(\zeta)>n$ for some $n\geq 2$.
The proofs, unless reasonably short, are carried
out in Section~\ref{profs}.

\section{Refinements of the equivalence principle} \label{refin}

Theorem~\ref{nabum} will be an immediate consequence of 
Theorem~\ref{juchu}, 
Theorem~\ref{unumber} and Theorem~\ref{tnumber} 
formulated below in this section. 

\subsection{Upper bounds for $\lambda_{n}$} \label{b}

The upper bounds in Theorem~\ref{nabum} and Theorem~\ref{genau}
are a consequence of the following very general Theorem~\ref{juchu}.
We agree on $w_{0}(\zeta)=0$.

\begin{theorem} \label{juchu}
Let $n\geq 1$ be an integer and $\zeta$ a transcendental real number.
Assume $w_{n}(\zeta)<\infty$. 
Then we have 
\begin{equation} \label{eq:vor1}
\lambda_{N}(\zeta)\leq \max\left\{ \frac{1}{\widehat{w}_{n}(\zeta)},\frac{1}{\widehat{w}_{N-n+1}(\zeta)-w_{n}(\zeta)} \right\}, 
\qquad\; N\geq  \lceil w_{n}(\zeta)\rceil+n-1.
\end{equation}
Moreover, in the case of $w_{n}(\zeta)<2n+1$ we have
\begin{equation} \label{eq:vor4}
\lambda_{N}(\zeta)\leq \max\left\{ \frac{1}{\widehat{w}_{n}(\zeta)},\frac{1}{\widehat{w}_{N-n+1}(\zeta)-w_{N-2n}(\zeta)} \right\}, 
\qquad \lfloor w_{n}(\zeta)\rfloor+n\leq N\leq 3n.
\end{equation}
\end{theorem}

We see that in case of $w_{n}(\zeta)<2n$, for $N<3n$ the bound
\eqref{eq:vor4} is possibly stronger than \eqref{eq:vor1} because of
the smaller index in the right expression. The case $w_{n}(\zeta)<n+1$
and $N=2n$ in \eqref{eq:vor4} will play a crucial role
for improving the upper bounds 
for the exponents $\widehat{\lambda}_{2n}(\zeta)$
in Section~\ref{getupper}. The estimate \eqref{eq:vor1} with
a suitable choice of $N$ yields the desired implications
for the equivalence principle.

\begin{corollary}  \label{juchurra}
Let $n\geq 1$ be an integer and $\zeta$ a transcendental real number
and assume $w_{n}(\zeta)<\infty$. Then
\begin{equation} \label{eq:vor2}
\lambda_{N}(\zeta)\leq \frac{1}{n}, \qquad \qquad N\geq \lceil w_{n}(\zeta)\rceil+2n-1.
\end{equation}
Thus, if $\zeta$ is not a $U$-number then $\lim_{n\to\infty} \lambda_{n}(\zeta)=0$, and if
$\zeta$ is an $S$-number then 
$\overline{\lambda}(\zeta)=\limsup_{n\to\infty} n\lambda_{n}(\zeta)<\infty$, and more precisely
\begin{equation} \label{eq:juchur}
\overline{\lambda}(\zeta) \leq \overline{w}(\zeta)+2, \qquad \underline{\lambda}(\zeta) \leq \underline{w}(\zeta)+2.
\end{equation}
\end{corollary}

\begin{proof}
In view of \eqref{eq:wmono}, as soon as $N\geq w_{n}(\zeta)+2n-1$
the right hand side
in \eqref{eq:vor1} can be estimated above by
\[
\max\left\{ \frac{1}{\widehat{w}_{n}(\zeta)},\frac{1}{\widehat{w}_{N-n+1}(\zeta)-w_{n}(\zeta)} \right\} \leq
\max \left\{ \frac{1}{n},\frac{1}{N-n+1-w_{n}(\zeta)}\right\rbrace 
=\frac{1}{n}.
\]
Hence \eqref{eq:vor2} follows.
The claim \eqref{eq:juchur} follows by reversing the argument.
For $\epsilon>0$ and large $N$, choose 
$n=\lceil N/(\overline{w}(\zeta)+2+\epsilon)\rceil$
and $n=\lceil N/(\underline{w}(\zeta)+2+\epsilon)\rceil$, respectively.
The condition in \eqref{eq:vor2} is satisfied and 
we obtain 
$\lambda_{N}(\zeta)\leq 1/n= (\overline{w}(\zeta)+2)/N+\varepsilon_{N}$
and 
$\lambda_{N}(\zeta)\leq 1/n= (\underline{w}(\zeta)+2)/N+\varepsilon_{N}$,
respectively, where $\varepsilon_{N}$ tends to $0$ $\epsilon$ does
and $N\to\infty$. It suffices to let $\epsilon\to 0$.
\end{proof}

If $w_{n}(\zeta)$ is not too large (considerably smaller than $2n$)
and for small $t$ the values $w_{t}(\zeta)$ 
do not exceed $t$ by much, then \eqref{eq:vor4}
yields smaller $N$ for the conclusion in \eqref{eq:vor2}.
On the other hand, the bound $1/n$ in \eqref{eq:vor2}
in general cannot be improved for any $N$,
as follows from Theorem~\ref{nabum} by taking $\zeta$ a $U_{n+1}$-number.

\subsection{Lower bounds for $\lambda_{n}$} \label{a}

To formulate the results of this section in full extent, 
we need to define successive minima exponents that refine
the classical exponents $w_{n}(\zeta)$ and $\lambda_{n}(\zeta)$. 
For $1\leq j\leq n+1$, let $\lambda_{n,j}(\zeta)$ 
and $\widehat{\lambda}_{n,j}(\zeta)$
be the supremum of $\lambda$ for 
which \eqref{eq:lambda} has $j$ linearly independent integer vector solutions for arbitrarily large $X$ and all large $X$, respectively.
Similarly, let $w_{n,j}(\zeta)$ and $\widehat{w}_{n,j}(\zeta)$ be the supremum of $w$ for which \eqref{eq:w} has
$j$ linearly independent polynomial solutions for arbitrarily large and all large $X$, respectively. 
Obviously, for $j=1$ we recover the corresponding classical exponents, 
and the relations
\begin{align*}
\lambda_{n,1}(\zeta)&\geq \lambda_{n,2}(\zeta)\geq \cdots\geq \lambda_{n,n+1}(\zeta),\qquad 
\widehat{\lambda}_{n,1}(\zeta)\geq \widehat{\lambda}_{n,2}(\zeta)\geq \cdots\geq \widehat{\lambda}_{n,n+1}(\zeta),  \\
w_{n,1}(\zeta)&\geq w_{n,2}(\zeta)\geq \cdots\geq w_{n,n+1}(\zeta),\qquad 
\widehat{w}_{n,1}(\zeta)\geq \widehat{w}_{n,2}(\zeta)\geq \cdots\geq \widehat{w}_{n,n+1}(\zeta),
\end{align*}
hold.

\begin{theorem}  \label{unumber}
Let $m\geq 2$ be an integer and $\zeta$ be a $U_{m}$-number. Then 
\begin{equation}  \label{eq:neuheit}
\lambda_{n}(\zeta)\geq \frac{1}{m-1}, \qquad n\geq 1.
\end{equation}
If and only if additionally $w_{m-1}(\zeta)=m-1$ holds, then 
\begin{equation} \label{eq:neueun}
\lambda_{n}(\zeta)=\lambda_{n,2}(\zeta)=\cdots=\lambda_{n,m}(\zeta)=\frac{1}{m-1}, \qquad n\geq m-1.
\end{equation}
If and only if moreover $w_{1}(\zeta)=w\in[1,2]$ and
$w_{t}=t$ for any $2\leq t\leq m-1$, then the sequence $(\lambda_{n}(\zeta))_{n\geq 1}$ is given by
\begin{equation} \label{eq:zzz}
\left(w,\frac{1}{2},\frac{1}{3},\ldots,\frac{1}{m-2},\frac{1}{m-1},\frac{1}{m-1},\frac{1}{m-1},\cdots\right).
\end{equation}
\end{theorem}

\begin{remark}
Any $U_{m}$-number $\zeta$
satisfies $\widehat{w}_{n}^{\ast}(\zeta)\leq m$ for all $n\geq 1$,
see~\cite[Corollary~2.5]{bschlei}.
Combining this with the well-known bound \eqref{eq:besides} below would
yield $\lambda_{n}(\zeta)\geq 1/m$ for any $U_{m}$-number $\zeta$ and 
$n\geq 1$, a weaker conclusion than \eqref{eq:neuheit}.
\end{remark}

\begin{remark}
Obviously $n=m-1$ is the smallest index for which \eqref{eq:neueun} can possibly hold by \eqref{eq:ldiri}.
Clearly \eqref{eq:neueun} extends reasonably
the claims of Theorem~\ref{nabum},
upon the strong assumption $w_{m-1}(\zeta)=m-1$.
\end{remark}

\begin{remark}
We notice that \eqref{eq:zzz} with $w=1$
coincides with the sequence $(\lambda_{n}(\zeta))_{n\geq 1}$ for any real algebraic number $\zeta$ of degree exactly $m$,
a well-known consequence of Schmidt's Subspace Theorem.
Hence we have a criterion when a real number behaves like an algebraic real number of given degree with respect to simultaneous approximation.
\end{remark}

If we agree on $1/0=\infty$ then \eqref{eq:neuheit} is true for $n=1$ as well, but has been observed in~\cite{bug}
as a consequence of \eqref{eq:herz}.
For $n\leq m$, the estimate \eqref{eq:neuheit} follows from Khintchine's 
inequality \eqref{eq:khintchine} and \eqref{eq:reihenfolge}, however
for $n>m$ the result is new.
A similar method as in the proof of Theorem~\ref{unumber} will lead to the next partial claim of Theorem~\ref{nabum}. 

\begin{theorem} \label{tnumber}
For any transcendental real $\zeta$
the quantities defined in \eqref{eq:defwoben}, \eqref{eq:defwunten} 
satisfy
\begin{equation} \label{eq:wiedefwoben}
\frac{(\overline{w}(\zeta)+1)^{2}}{4\overline{w}(\zeta)} \leq
\overline{\lambda}(\zeta), \qquad \frac{(\underline{w}(\zeta)+1)^{2}}{4\underline{w}(\zeta)} \leq
\underline{\lambda}(\zeta).
\end{equation}
In particular, any $T$-number $\zeta$ satisfies
\begin{equation} \label{eq:altheit}
\limsup_{n\to\infty} n\lambda_{n}(\zeta)=\infty.
\end{equation}
\end{theorem} 

We close this section with a variant of Theorem~\ref{tnumber}
for uniform exponents, and for sake of completeness
also add consequences of 
Theorem~\ref{difizil} and from~\cite{ichrelations}.
Define the quantities
\begin{equation} \label{eq:q1}
\overline{\widehat{w}}(\zeta):= 
\limsup_{n\to\infty} \frac{\widehat{w}_{n}(\zeta)}{n}, \qquad\qquad \overline{\widehat{\lambda}}(\zeta):= 
\limsup_{n\to\infty} n\widehat{\lambda}_{n}(\zeta), 
\end{equation}
and
\begin{equation} \label{eq:q2}
\underline{\widehat{w}}(\zeta):= 
\liminf_{n\to\infty} \frac{\widehat{w}_{n}(\zeta)}{n}, \qquad\qquad \underline{\widehat{\lambda}}(\zeta):= \liminf_{n\to\infty} n\widehat{\lambda}_{n}(\zeta). 
\end{equation}
They satisfy $1\leq \underline{\widehat{w}}(\zeta)
\leq \overline{\widehat{w}}(\zeta)\leq 2$ and 
$1\leq \underline{\widehat{\lambda}}(\zeta)
\leq \overline{\widehat{\lambda}}(\zeta)\leq 2$ by \eqref{eq:wmono},
\eqref{eq:ldiri}, \eqref{eq:bound2} and
the estimate $\widehat{\lambda}_{n}(\zeta)\leq 2/n$,
as established
in~\cite[Theorem~2a]{davsh},~\cite{laurent},~\cite{ichrelations}
reproduced in Section~\ref{widhat} below. 
Moreover we have
\begin{equation} \label{eq:vergleich}
\overline{\widehat{w}}(\zeta)-1\leq \overline{\widehat{w}}^{\ast}(\zeta)\leq \overline{\widehat{w}}(\zeta), \qquad 
\underline{\widehat{w}}(\zeta)-1\leq \underline{\widehat{w}}^{\ast}(\zeta)\leq \underline{\widehat{w}}(\zeta)
\end{equation}
again by \eqref{eq:bound2}.

\begin{theorem}   \label{uniformiert}
Let $\zeta$ be a transcendental real number. Then 
the above defined exponents satisfy
\begin{equation} \label{eq:family}
\frac{(\overline{\widehat{w}}(\zeta)+1)^{2}}
{4\overline{\widehat{w}}(\zeta)} \leq
\overline{\widehat{\lambda}}(\zeta)\leq 1+\frac{1}{\underline{\widehat{w}}(\zeta)}, 
\qquad\qquad \frac{(\underline{\widehat{w}}(\zeta)+1)^{2}}{4\underline{\widehat{w}}(\zeta)} \leq
\underline{\widehat{\lambda}}(\zeta)\leq 1+\frac{1}{\overline{\widehat{w}}(\zeta)},
\end{equation}
and
\begin{equation} \label{eq:day}
\overline{\widehat{w}}(\zeta)
\leq \min\left\{ 2,\overline{w}(\zeta),1+
\frac{4}{\underline{w}(\zeta)}\right\}, \qquad\qquad
\underline{\widehat{w}}(\zeta)
\leq \min\left\{ 2,\underline{w}(\zeta), 1+
\frac{4}{\overline{w}(\zeta)}\right\}.
\end{equation}

\end{theorem}

The particular consequence that $\overline{w}(\zeta)=\infty$ implies
$\underline{\widehat{w}}(\zeta)=1$ was already noticed 
in~\cite[Corollary~2.5]{bschlei}. From \eqref{eq:family} 
we can also deduce that
$\overline{\widehat{w}}(\zeta)>1$ implies
$\overline{\widehat{\lambda}}(\zeta)>1$, and $\underline{\widehat{w}}(\zeta)>1$ implies
$\underline{\widehat{\lambda}}(\zeta)>1$. 
We believe that the converse implications hold as well.
German~\cite{german} established refinements of the transference
principle \eqref{eq:khintchine} for the uniform exponents, 
however they are again insufficient for such problems.
It would be nice to include the exponents
on $\overline{w}^{\ast}(\zeta)$ and $\underline{w}^{\ast}(\zeta)$
in the picture, related to the Wirsing problem. However, we do not
know what to conjecture. We remark that from the sparse 
present results on the 
exponents $\widehat{w}_{n}, \widehat{\lambda}_{n}$, 
we cannot exclude that the quantities 
in \eqref{eq:q1}, \eqref{eq:q2} 
all equal $1$ for any transcendental real number $\zeta$.

\section{Applications: Metric theory and spectra} \label{metric}

\subsection{The joint spectrum of $(\lambda_{n})_{n\geq 1}$} \label{spectrum}

We study the set of sequences
$\{(\lambda_{n}(\zeta))_{n\geq 1}:\zeta\in\mathbb{R}\}$,
which we will refer to as the {\em joint spectrum 
of $(\lambda_{n}(\zeta))_{n\geq 1}$}. 
Bugeaud~\cite[Theorem~4]{bug} showed the existence of transcendental real $\zeta$ such that
$\lambda_{n}(\zeta)=1$ for all $n\geq 1$. Thus the constant $1$ sequence
belongs to the joint spectrum. 
Moreover~\cite[Theorem~5]{bug} 
asserted that for given $\lambda\in[1,3]$ there exists transcendental 
real $\zeta$ such that $\lambda_{1}(\zeta)=\lambda$ and $\lambda_{2}(\zeta)=1$. 
Both claims are sharp in some sense. Indeed, in view of
\begin{equation} \label{eq:herz}
\lambda_{nk}(\zeta)\geq \frac{\lambda_{k}(\zeta)-n+1}{n}, 
\qquad\qquad k\geq 1,\; n\geq 1,
\end{equation}
from~\cite[Lemma~1]{bug}, for $k=2, n=1$ we see that $\lambda_{1}(\zeta)\leq 3$ 
when $\lambda_{2}(\zeta)=1$.
A conjectured generalization of \eqref{eq:herz} 
from~\cite{schlei} was rephrased in Section~\ref{brief}. 
As pointed out,
there is equality in \eqref{eq:herz} if $\lambda_{nk}(\zeta)>1$, in particular 
\begin{equation} \label{eq:trampel}
\lambda_{n}(\zeta)=\frac{\lambda_{1}(\zeta)-n+1}{n},  \qquad \qquad \text{if} \; \lambda_{n}(\zeta)>1.
\end{equation}
Hence we cannot have $\lambda_{n}(\zeta)>1$ for all $n\geq 1$,
unless $\zeta$ is a Liouville number, that is $\lambda_{1}(\zeta)=\infty$,
and in this case the joint spectrum is the constant $\infty$ sequence
by \eqref{eq:herz} as observed in~\cite[Corollary~2]{bug}. 
The identity \eqref{eq:trampel} implied a negative answer 
on Bugeaud's~\cite[Problem~2]{bug}
where he proposed that the conditions 
\eqref{eq:ldiri}, \eqref{eq:reihenfolge} and \eqref{eq:herz}
might be the only limitations for the joint spectrum
of $(\lambda_{n}(\zeta))_{n\geq 1}$. Our 
Theorem~\ref{nabum} and Corollary~\ref{lamkor}
clearly again show that this is far from being true.

As a consequence of Theorem~\ref{unumber}, we 
determine the joint spectrum 
of $(\lambda_{n}(\zeta))_{n\geq 1}$ among $U_{2}$-numbers $\zeta$,
thereby among all $\zeta$ satisfying $\lambda_{n}(\zeta)>1/2$ 
for all $n\geq 1$ by Theorem~\ref{nabum}.

\begin{theorem} \label{besser}
Let $\zeta$ be a $U_{2}$-number with $w_{1}(\zeta)=w\in[1,\infty)$. Then 
\begin{align}
\lambda_{n}(\zeta)&= \frac{w+1-n}{n}, \qquad \qquad\qquad 1\leq n\leq \frac{w+1}{2},   \label{eq:1} \\
\lambda_{n}(\zeta)&= \lambda_{n,2}(\zeta)=1, \qquad\qquad\qquad n\geq \frac{w+1}{2}.   \label{eq:2}
\end{align}
In particular if $w=1$ then $\lambda_{n}(\zeta)=1$ for all $n\geq 1$. 
The sequences of the form
\begin{equation} \label{eq:3}
\left(w,\frac{w-1}{2},\frac{w-2}{3},\ldots,\frac{w+1-\lfloor \frac{w+1}{2}\rfloor}{\lfloor \frac{w+1}{2}\rfloor}, 1,1,1,\ldots \right), \qquad w\geq 1,
\end{equation}
coincide precisely with the sequences $(\lambda_{n}(\zeta))_{n\geq 1}$ induced by the set of $U_{2}$-numbers $\zeta$.
In particular they all belong to the joint spectrum of $(\lambda_{n})_{n\geq 1}$.
Conversely, the sequences in \eqref{eq:3} with $w\in[1,\infty]$ are precisely those sequences in the joint spectrum of $(\lambda_{n})_{n\geq 1}$
with $\lambda_{n}(\zeta)>\frac{1}{2}$ for all $n\geq 1$.
\end{theorem}

The claims vastly generalize both~\cite[Theorem~4 and Theorem~5]{bug} mentioned above.
In the proof we will use the existence of $U_{2}$-numbers with any prescribed value $w_{1}(\zeta)\in[1,\infty)$.
We point out that more generally 
Alnia\c{c}ik~\cite{alni} essentially constructed $U_{n}$-numbers $\zeta$ with prescribed value of $w_{1}(\zeta)\in[1,\infty)$,
for any $n\geq 2$ (although he only explicitly stated
the case $w_{1}(\zeta)=1$ in~\cite{alni}).
See also~\cite{alu} for $U$-numbers with small transcendence degree.
However, the existence of $U_{n}$-numbers which satisfy the hypothesis $w_{n-1}(\zeta)=n-1$, let alone the more general hypothesis,
in Theorem~\ref{unumber} is open for $n\geq 3$, which among other things prevents us from generalizing Theorem~\ref{besser}.
Observe that even the strong hypothesis for \eqref{eq:zzz}
would be covered by 
an affirmative answer to Problem~\ref{mp}.

\subsection{Upper bounds for $\widehat{\lambda}_{n}(\zeta)$} \label{getupper} 
 
Assume $n\geq 1$ is an integer and $\zeta$ a 
transcendental real number with the property $w_{n}(\zeta)<n+1$.
If we let $N=2n$, as a consequence of~\eqref{eq:vor4} we obtain
\begin{equation} \label{eq:vor3}
\lambda_{2n}(\zeta)\leq \frac{1}{\widehat{w}_{n}(\zeta)} \leq \frac{1}{n}.
\end{equation}
Upon the assumption $w_{n}(\zeta)<n+1$, the classical
estimates \eqref{eq:reihenfolge} and \eqref{eq:khintchine} 
would only yield
$\lambda_{2n}(\zeta)\leq \lambda_{n}(\zeta)<\frac{2}{n}$,
so \eqref{eq:vor3} yields an improvement by the factor $2$.
We can use the conditional result
\eqref{eq:vor3} to sharpen the best known upper bound for the 
exponent $\widehat{\lambda}_{n}(\zeta)$ for even $n$. 
The problem on determining such bounds
dates back to Davenport and Schmidt~\cite{davsh} who
established a relation to approximation to real numbers 
by algebraic integers, connected to Wirsing's Problem~\cite{wirsing}.
Their original result has been refined for odd $n$ by
Laurent~\cite{laurent}, who showed 
$\widehat{\lambda}_{2n}(\zeta)\leq\widehat{\lambda}_{2n-1}(\zeta)\leq n^{-1}$.
A significantly shorter
proof of this bound together
with a slight refinement of the bound for even $n$
was recently given by the author~\cite[Theorem~2.3]{ichrelations}. 
Our new refinement
for even $n$ is again based on~\cite[Theorem~2.1]{ichrelations}. 
It asserts that
for $m,n$ positive integers and $\zeta$ any 
transcendental real number, the estimate
\begin{equation} \label{eq:ribiza}
\widehat{\lambda}_{m+n-1}(\zeta)\leq\max \left\{\frac{1}{w_{m}(\zeta)},
\frac{1}{\widehat{w}_{n}(\zeta)}\right\}
\end{equation}
holds. The specification $m=n$ directly led to Laurent's
estimate quoted above.
The addition of~\eqref{eq:vor3} leads to a better bound. The variable $n$ in \eqref{eq:ribiza} will correspond 
to $n+1$ in the proof of following Theorem~\ref{unifthm}.

\begin{theorem} \label{unifthm}
Let $n\geq 1$ be an integer and $\zeta$ a transcendental real number. 
Then we have 
\begin{equation} \label{eq:tropf}
\widehat{\lambda}_{2n}(\zeta)\leq  
\sqrt{\left( n+\frac{1}{2n}\right)^{2}-\frac{1}{n}}-n+\frac{1}{2n}.
\end{equation}
In the case of $\lambda_{2n}(\zeta)>\frac{1}{n}$, the stronger
bound $\widehat{\lambda}_{2n}(\zeta)\leq \frac{1}{n+1}$ holds.
\end{theorem}

\begin{proof}
The estimate \eqref{eq:ribiza} with proper choices of integer
parameters yields
\[
\widehat{\lambda}_{2n}(\zeta)\leq 
\max\left\{ \frac{1}{w_{n}(\zeta)}, 
\frac{1}{\widehat{w}_{n+1}(\zeta)}\right\}.
\]
In the case of $w_{n}(\zeta)\geq n+1$, by \eqref{eq:wmono} we infer 
$\widehat{\lambda}_{2n}(\zeta)\leq (n+1)^{-1}$, which is smaller
than the right hand side in \eqref{eq:tropf}. 
In case of $w_{n}(\zeta)< n+1$,
we may apply~\eqref{eq:vor3}. We insert this value
$\lambda_{2n}(\zeta)= n^{-1}$ in the reformulation
\[
\widehat{\lambda}_{2n}(\zeta)\leq 
-\frac{2n-2+(2n-1)\lambda_{2n}(\zeta)}{2}+
\sqrt{ \left(\frac{2n-2+(2n-1)\lambda_{2n}(\zeta)}{2}\right)^{2}+(2n-1)\lambda_{2n}(\zeta)}
\]
of Schmidt and Summerer~\cite[(1.21)]{ssch}, and 
elementary rearrangements lead to \eqref{eq:tropf}. Reversing the
proof we see that $\lambda_{2n}(\zeta)>\frac{1}{n}$
implies $w_{n}(\zeta)\geq n+1$, and as above we infer the bound 
$\widehat{\lambda}_{2n}(\zeta)\leq (n+1)^{-1}$.
\end{proof}

The bound in \eqref{eq:tropf} is of asymptotic order
$\frac{1}{n}-\frac{1}{2n^{2}}+O(n^{-3})$.
We obtain a reasonable improvement to
the old bound $\widehat{\lambda}_{2n}(\zeta)$ 
of order $\frac{1}{n}-\frac{1}{2n^{3}}+O(n^{-4})$
in~\cite[Theorem~2.3]{ichrelations}. We explicitly
state the new estimates for $n=1$ and $n=2$, which read
\[
\widehat{\lambda}_{2}(\zeta)\leq 
\frac{\sqrt{5}-1}{2}=0.6180\ldots, \qquad
\widehat{\lambda}_{4}(\zeta)\leq 
\frac{\sqrt{73}-7}{4}=0.3860\ldots.
\]
The bound for $\widehat{\lambda}_{2}(\zeta)$ is well-known to be sharp
as shown by Roy~\cite{roy}. Roy~\cite{damroy} also established 
the bound
$\widehat{\lambda}_{3}(\zeta)\leq (2+\sqrt{5}-\sqrt{7+2\sqrt{5}})/2=
0.4245\ldots$,  which previously represented the best known 
bound for $\widehat{\lambda}_{4}(\zeta)$ as well.
Our new bound for $\widehat{\lambda}_{4}(\zeta)$ is finally
smaller. Improvements
of \eqref{eq:tropf} can be made  for $n\geq 2$, conditional
on the conjecture
of Schmidt and Summerer proposed in~\cite[page~92]{sums}
concerning the minimum value of 
the quotient $\lambda_{n}(\zeta)/\widehat{\lambda}_{n}(\zeta)$
in terms of $\widehat{\lambda}_{n}(\zeta)$. 
Concretely,  applying~\cite[(30)]{ichjp}
we obtain as an conditional upper bound  
for $\widehat{\lambda}_{2n}(\zeta)$ the implicit solution 
$\lambda=\lambda(n)$ of the polynomial equation
\[
n^{2n} \lambda^{2n+1}-(n+1)\lambda +1=0,
\]
in the interval $(\frac{1}{n+1}, \frac{1}{n})$.
It can be shown that this value $\lambda$ is of the 
form $\frac{1}{n}-\frac{\alpha}{n^{2}}+O(n^{-3})$, 
for $\alpha\in(0.796,0.797)$ the unique positive
real root of the power series
\[
-1+\sum_{k=1}^{\infty} \frac{(-2)^{k+1}}{(k+1)!}x^{k}=
-1+2x-\frac{4}{3}x^{2}+\frac{2}{3}x^{3}
-\frac{4}{15}x^{4}+\frac{4}{45}x^{5}-\cdots.
\]
%
The resulting conditional numerical bounds for 
some small $n$ can be computed as 
\[
\widehat{\lambda}_{4}(\zeta)\leq 0.3706\ldots
, \quad \widehat{\lambda}_{6}(\zeta)\leq 0.2681\ldots,
\quad \widehat{\lambda}_{20}(\zeta)\leq 0.0928\ldots.
\]
In comparison,
the unconditional bounds from \eqref{eq:tropf} are numerically given by $\widehat{\lambda}_{4}(\zeta)\leq 0.3860\ldots$,
$\widehat{\lambda}_{6}(\zeta)\leq 0.2803\ldots$
and $\widehat{\lambda}_{20}(\zeta)\leq 0.0950\ldots$.

\subsection{On the case $\widehat{w}_{n}(\zeta)>n$} \label{widhat}

Using Schmidt's Subspace Theorem, 
Adamczewski and Bugeaud~\cite{adambug} found explicit upper bounds
for the exponents $w_{m}(\zeta)$ for numbers $\zeta$ that admit
many very small polynomial evaluations $\vert P_{i}(\zeta)\vert$ at
$P_{i}\in\mathbb{Z}[T]$ of bounded degree and high rate in the sense 
of $\log H(P_{i+1})/\log H(P_{i})$ being absolutely bounded. 
See also A. Baker~\cite{abaker} for earlier results in the case $n=1$.
In~\cite[Section~5]{adambug} they provided types of numbers
that fall into this category.
Any number that satisfies 
\begin{equation}  \label{eq:whut}
\widehat{w}_{n}(\zeta)>n, \qquad\qquad \text{for some}\;\; n\geq 2, 
\end{equation}
and is not a $U_{m}$-number for some $m\leq n$, 
has the desired property. In fact we can exclude the case $m=n$
by~\cite[Corollary~2.5]{bschlei}, and $m=1$ as well
by~\cite[Theorem~1.12]{schlei}. So there is no additional 
condition when $n=2$.
For such numbers 
they established an exponential bound of the form
\begin{equation} \label{eq:adb}
w_{m}^{\ast}(\zeta)\leq \exp (c\cdot(\log 3m)^{n} (\log\log 3m)^{n}),
\qquad m\geq n+1,
\end{equation}
where $c=c(\zeta)>0$ is some ineffective 
constant~\cite[Theorem~4.2,~5.3]{adambug}.
In particular numbers that satisfy \eqref{eq:whut}
cannot be $U_{m}$-numbers
for $m>n$.  If we replace \eqref{eq:whut} by the (at least formally)
stronger condition
\begin{equation} \label{eq:wsternhut}
\widehat{w}_{n}^{\ast}(\zeta)>n, \qquad\qquad \text{for some}\;\; n\geq 2, 
\end{equation}
the same conclusion \eqref{eq:adb} holds without 
any additional condition by~\cite[Theorem~2.4]{bschlei}.
It is probable that the 
condition \eqref{eq:whut}
in fact implies $\widehat{w}_{n}(\zeta)=\widehat{w}_{n}^{\ast}(\zeta)$,
Bugeaud recently posed 
the case $n=2$ as a problem~\cite[Problem~2.9.7]{bdraft}. 
Another class of numbers $\zeta$ satisfying the property
are Champerowne-type
numbers whose expansion in some base $b\geq 2$ is
of the form $\zeta=\zeta_{b,P}=0.(P(1))_{b}(P(2))_{b}\ldots$, where
$P\in\mathbb{Z}[T]$ is a non-constant polynomial
and $(P(h))_{b}$ is the integer $h$ written in base $b$. 
The classical Champerowne number is obtained 
for $b=10$ and $P(T)=T$.
We have 
\begin{equation} \label{eq:champ}
w_{m}^{\ast}(\zeta_{b,P})\leq (2m)^{c^{\prime}\cdot \log\log 3m}, \qquad\qquad m\geq 1,
\end{equation}
where $c^{\prime}=c^{\prime}(\zeta_{b,P})$ is again 
a suitable constant~\cite[Theorem~3.1,~5.1]{adambug}.
See~\cite[Section~5]{adambug} for more examples.
Corollary~\ref{juchurra} combined with \eqref{eq:adb} and \eqref{eq:champ}
yields an estimate for the minimum decay of the 
exponents $\lambda_{N}(\zeta)$ for large $N$, in the case of
the Champerowne-type numbers by roughly some (ineffective) negative 
power of $N$.

\begin{corollary} \label{landak}
Let $\zeta$ be a real number and $\varepsilon>0$
arbitrarily small. First assume that either
\eqref{eq:whut} and if $n\geq 3$
additionally $w_{n-1}(\zeta)<\infty$ holds, 
or \eqref{eq:wsternhut} holds. Then
there exists a constant $d=d(\zeta,\varepsilon)>0$ so that
\begin{equation} \label{eq:gbush}
\lambda_{N}(\zeta)\leq \exp (-d (\log N)^{\frac{1-\varepsilon}{n}}), 
\qquad\qquad N\geq 1.
\end{equation}
For $\zeta=\zeta_{b,P}$ any Champerowne-type number we 
have the stronger decay
\begin{equation} \label{eq:gwbush}
\lambda_{N}(\zeta)\leq \exp (-d^{\prime} (\log N)^{1-\varepsilon}), 
\qquad\qquad N\geq 1,
\end{equation}
for suitable $d^{\prime}=d^{\prime}(\zeta,\varepsilon)>0$.
In particular in both cases $\lim_{N\to\infty} \lambda_{N}(\zeta)=0$.
\end{corollary}

\begin{proof}
We show \eqref{eq:gbush}.
In the preceding comments we pointed out that
\eqref{eq:adb} is satisfied for $\zeta$ that satisfies
any of the stated assumptions.
Combining this with \eqref{eq:duzu} and some crude estimates imply
$w_{m}(\zeta)\leq Y:=\exp (c^{\prime}(\log m)^{n(1+\varepsilon)})$
with some $c^{\prime}=c^{\prime}(\zeta,\varepsilon)$ possibly 
slightly larger than $c$. If 
we let $N=\lceil Y\rceil+2m$, from \eqref{eq:vor2} we obtain
$\lambda_{N}(\zeta)\leq 1/m$. Hence, for large $N$ and suitable $d$, 
the claim \eqref{eq:gbush} follows from
elementary estimates and rearrangements. 
Since $\zeta$ is not a $U_{1}$-number and thus 
$\lambda_{N}(\zeta)\leq \lambda_{1}(\zeta)<\infty$ for all $N\geq 1$,
by increasing $d$ if necessary we may consider any $N\geq 1$.
The claim \eqref{eq:gwbush} for Champerowne-type numbers
follows in a similar way from \eqref{eq:champ} and 
Corollary~\ref{juchurra}.
\end{proof}

Upon very similar assumptions as for \eqref{eq:adb}, 
for integers $m$ not exceeding some bound,
significantly smaller upper bounds for $w_{m}(\zeta)$ were established 
in~\cite{bschlei}.
Indeed~\cite[Theorem~2.4]{bschlei} can be reformulated in the
following way.
Upon the assumption $\widehat{w}_{n}^{\ast}(\zeta)-(n+u-1)>0$ 
for some integer $u\geq 1$, we have
\begin{equation} \label{eq:sturz}
w_{n+j}(\zeta) \leq 
\frac{(n-1)\widehat{w}_{n}^{\ast}(\zeta)}
{\widehat{w}_{n}^{\ast}(\zeta)-(n+j)}, \qquad\qquad 0\leq j\leq u-1.
\end{equation}
Similarly, if $\widehat{w}_{n}(\zeta)-(n+u-1)>0$
and additionally
\begin{equation} \label{eq:uaah}
w_{n+u-1}(\zeta)>w_{n-1}(\zeta), \qquad
\text{or} \qquad \widehat{w}_{n}(\zeta)=\widehat{w}_{n}^{\ast}(\zeta)
\end{equation}
holds, then~\cite[Theorem~2.2]{bschlei} analogously asserts
\begin{equation} \label{eq:salsa}
w_{n+j}(\zeta) \leq 
\frac{(n-1)\widehat{w}_{n}(\zeta)}
{\widehat{w}_{n}(\zeta)-(n+j)}, \qquad\qquad 0\leq j\leq u-1.
\end{equation}
In particular $w_{n+u-2}(\zeta)<(n-1)(2n-1)=2n^{2}-3n+1$ 
if $u\geq 2$ by \eqref{eq:bound2}. The estimate
for $w_{n+u-1}(\zeta)$ is still 
reasonably good unless $\widehat{w}_{n}^{\ast}(\zeta)$ 
(or $\widehat{w}_{n}(\zeta)$) is very close to $n+u-1$.
As another new contribution, by combining results from 
Section~\ref{refin} we infer upper bounds for
the next larger exponent $w_{n+u}(\zeta)$. They turn out
to be better than~\eqref{eq:adb} for $m=n+u$,
upon a stronger assumption.
We start with the most general version and specify below.

\begin{theorem} \label{bugadam}
Let $\zeta$ be a transcendental real number and $n\geq 2$ and
$u\geq 1$ be integers.
If $\widehat{w}_{n}^{\ast}(\zeta)>n+u-1$
is satisfied, then we have
\begin{equation} \label{eq:unkon}
w_{n+u}(\zeta) \leq 
\frac{ \widehat{w}_{n}^{\ast}(\zeta)^{3}-u\widehat{w}_{n}^{\ast}(\zeta)^{2}+(nu+n-1)\widehat{w}_{n}^{\ast}(\zeta)-n^{2}}
{(\widehat{w}_{n}^{\ast}(\zeta)-n)(\widehat{w}_{n}^{\ast}(\zeta)-n-u+1)}.
\end{equation}
If we assume $\widehat{w}_{n}(\zeta)>n+u-1$ and additionally
\eqref{eq:uaah} holds, then we have
\begin{equation} \label{eq:tja}
w_{n+u}(\zeta) \leq 
\frac{ \widehat{w}_{n}(\zeta)^{3}-u\widehat{w}_{n}(\zeta)^{2}+(nu+n-1)\widehat{w}_{n}(\zeta)-n^{2}}
{(\widehat{w}_{n}(\zeta)-n)(\widehat{w}_{n}(\zeta)-n-u+1)}.
\end{equation}
In particular, if we have either
$\widehat{w}_{n}^{\ast}(\zeta)-(n+u-1)=:\delta_{1}>0$,
or 
$\widehat{w}_{n}(\zeta)-(n+u-1)=:\delta_{2}>0$ and \eqref{eq:uaah},
then for some effectively computable constant $c>0$ we have
\begin{equation} \label{eq:asyms}
w_{n+u}(\zeta) \leq c \cdot \frac{n^{3}}{\delta_{i}^{2}}, 
\end{equation}
for $i=1$ and $i=2$ respectively. In the case of $u\geq 2$ 
we have $w_{n+1}(\zeta) \leq c^{\prime}n^{3}/\delta_{i}$.
\end{theorem}

We highlight the case $u=1$, where the involved
conditions become more natural and the resulting bounds
can be rearranged to a nicer form as well.

\begin{corollary} \label{bugadamc}
Let $\zeta$ be a transcendental real number
and assume \eqref{eq:wsternhut} holds for some integer $n\geq 2$.
Then we have
\begin{equation} \label{eq:unkonp}
w_{n+1}(\zeta) \leq 
\frac{\widehat{w}_{n}^{\ast}(\zeta)^{3}-\widehat{w}_{n}^{\ast}(\zeta)}
{(\widehat{w}_{n}^{\ast}(\zeta)-n)^{2}}-1.
\end{equation}
If we assume the weaker condition \eqref{eq:whut} and additionally
\eqref{eq:uaah} for $u=1$, then similarly
\begin{equation} \label{eq:tjap}
w_{n+1}(\zeta) \leq \frac{\widehat{w}_{n}(\zeta)^{3}-
\widehat{w}_{n}(\zeta)}
{(\widehat{w}_{n}(\zeta)-n)^{2}}-1.
\end{equation}
In particular, if either 
$\widehat{w}_{n}^{\ast}(\zeta)-n=:\delta_{1}>0$ holds
or $\widehat{w}_{n}(\zeta)-n=:\delta_{2}>0$ and \eqref{eq:uaah}
holds, then we have $w_{n+1}(\zeta) < 8n^{3}/\delta_{i}^{2}$
for $i=1$ and $i=2$ respectively.
\end{corollary}

\begin{proof}
Let $u=1$ in Theorem~\ref{bugadam} and rearrange the right hand side
to obtain \eqref{eq:unkonp} and \eqref{eq:tjap}. The factor
$8n^3$ reflects a crude estimate of the nominator
using \eqref{eq:bound2}.
\end{proof}

We expect that the unpleasant condition 
$w_{n}(\zeta)>w_{n-1}(\zeta)$ for \eqref{eq:tjap} can be dropped, which would yield an unconditional
improvement of \eqref{eq:adb} for $m=n+1$.
We can prove this to be true in the case $n=2$.

\begin{corollary} \label{wk}
Let $\zeta$ be a real number that satisfies $\widehat{w}_{2}(\zeta)>2$.
Then we have
\[
w_{3}(\zeta)\leq 
\frac{ \widehat{w}_{2}(\zeta)^{3}-\widehat{w}_{2}(\zeta)^{2}+3\widehat{w}_{2}(\zeta)-4}{(\widehat{w}_{2}(\zeta)-2)^{2}}\leq 
\frac{d}{(\widehat{w}_{2}(\zeta)-2)^{2}},
\]
where we may choose $d=14.9444$.
\end{corollary}

\begin{proof}
Application of \eqref{eq:tjap} for $n=2$ yields the left inequality. 
We have to check that 
$\widehat{w}_{2}(\zeta)>2$ implies its condition 
$w_{1}(\zeta)<w_{2}(\zeta)$.
In~\cite[Theorem~2.5]{bschlei} it was shown
that $\min\{ w_{m}(\zeta), \widehat{w}_{n}(\zeta)\}\leq m+n-1$
holds for any positive integers $m,n$ and transcendental real $\zeta$.
Indeed, with $m=1$ and $n=2$, since $\widehat{w}_{2}(\zeta)>2$ 
this is only possible if 
$w_{1}(\zeta)\leq 2<\widehat{w}_{2}(\zeta)\leq w_{2}(\zeta)$.
Finally the numeric constant can be derived 
from $\widehat{w}_{2}(\zeta)\leq (3+\sqrt{5})/2$, see
Davenport and Schmidt~\cite{davsh}.
\end{proof}

A special subclass of numbers with the property $\widehat{w}_{2}(\zeta)>2$ are Sturmian continued fractions, see~\cite{buglaur}.
For such numbers Corollary~\ref{landak} and Corollary~\ref{wk} apply.
Previously it was not even known that for such numbers
$\lim_{N\to\infty} \lambda_{N}(\zeta)=0$ holds. On the other hand,
the exponent $w_{3}(\zeta)$ for Sturmian continued fractions
has been explicitly determined, from~\cite{buglaur} 
and~\cite[Theorem~2.1]{ichstw} we know that
$w_{2}(\zeta)=w_{3}(\zeta)=\widehat{w}_{2}(\zeta)/(\widehat{w}_{2}(\zeta)-2)$. However, there are many more numbers that satisfy
$\widehat{w}_{2}(\zeta)>2$, see for example Roy~\cite{roydicht}, for
which Corollary~\ref{wk} provides the first explicit upper
bounds for $w_{3}(\zeta)$ in terms of $\widehat{w}_{2}(\zeta)$.
On the other hand, no number with the property
\eqref{eq:whut} for some $n\geq 3$ is known.

\subsection{Metric theory}

Now we turn to the metric problem
of determining the Hausdorff dimensions
\[
h_{N}^{\lambda}= 
\dim(H_{N}^{\lambda}), \qquad \qquad
H_{N}^{\lambda}:=\{ \zeta\in\mathbb{R}: \lambda_{N}(\zeta)\geq \lambda\}
\]
posed in~\cite[Problem~2]{bug}. We use the subscript index $N$ instead 
of $n$ to avoid confusion in the proofs later.
Obviously the values $h_{N}^{\lambda}$ decay in both variables
$N,\lambda$. Furthermore,
$0\leq h_{N}^{\lambda}\leq 1$ for all $N$ and $\lambda$, and 
$h_{N}^{\lambda}=1$ when $\lambda\leq 1/N$ by \eqref{eq:ldiri}. 
Usually one is interested in the values $h_{N}^{\lambda}$ as a function of $\lambda\in[1/N,\infty]$ for fixed $N$.
For parameters greater than one, as a consequence 
of \eqref{eq:trampel} and the one-dimensional formula
by Jarn\'ik~\cite{jar} it was shown in~\cite[Corollary~1.8]{schlei} that 
\begin{equation} \label{eq:altegl}
h_{N}^{\lambda}=\frac{2}{(1+\lambda)N}, \qquad\qquad N\geq 1,\;\lambda> 1.
\end{equation}
For $N=2$, the problem is solved for parameters $\lambda\leq 1$  as well, Beresnevich, Dickinson and Velani~\cite{vel} showed that
\begin{equation} \label{eq:bervel}
h_{2}^{\lambda}= \frac{2-\lambda}{1+\lambda}, 
\qquad\qquad \frac{1}{2}\leq \lambda\leq 1.
\end{equation}
For $N\geq 3$ and $\lambda\leq 1$, the problem of 
determining $h_{N}^{\lambda}$ is open. The lower bound
\begin{equation} \label{eq:altegle}
h_{N}^{\lambda}\geq \frac{2}{(1+\lambda)N}, \qquad\qquad N\geq 1,\;\lambda\geq \frac{1}{N},
\end{equation}
follows from \eqref{eq:herz} as noticed in~\cite{bug}.
For $\lambda$ close to $1/N$, 
Beresnevich~\cite[Theorem~7.2]{beres} showed
\begin{equation}  \label{eq:beresn}
h_{N}^{\lambda} \geq \frac{N+1}{1+\lambda}-(N-1), 
\qquad\qquad \frac{1}{N}\leq \lambda< \frac{3}{2N-1}.
\end{equation}
He expects equality in the given interval. 
The upper bounds in \eqref{eq:bervel} clearly hold for
$N\geq 3$ and $\lambda\in[1/2,1]$ as well, however no improvement 
has been established yet and for $\lambda<1/2$ 
in fact nothing is known.
For technical reasons we also introduce the 
auxiliary variations
\[
g_{N}^{\lambda}= \dim(G_{N}^{\lambda}), \qquad\qquad
G_{N}^{\lambda}:= \{ \zeta\in\mathbb{R}: \lambda_{N}(\zeta)>\lambda\},
\]
of $H_{N}^{\lambda}, h_{N}^{\lambda}$. Clearly $h_{N}^{\lambda+\epsilon}\leq g_{N}^{\lambda}\leq h_{N}^{\lambda}$ for any $N,\lambda$ and $\epsilon>0$.
We in fact expect $g_{N}^{\lambda}=h_{N}^{\lambda}$ for all $N,\lambda$,
but this seems not to be completely obvious.

Essentially
by \eqref{eq:leftie} and Bernik's formula \eqref{eq:bern}, 
we can refine the lower bounds in \eqref{eq:altegle}
and establish non-trivial upper bounds.
We formulate several variants of new results.
First, comparable to \eqref{eq:beresn}, 
we estimate $h_{N}^{\lambda}$ 
for parameters $\lambda$ in a fixed 
ratio with the trivial lower bound $1/N$,
asymptotically for large $N$.

\begin{theorem}  \label{dadthm}
Let $\widetilde{\lambda}\geq 1$ be a parameter, and for $N\geq 1$ let 
$\theta_{N}=\widetilde{\lambda}\cdot\frac{1}{N}$. Then we have
\begin{equation} \label{eq:neuy}
h_{N}^{\theta_{N}} \geq 
\frac{1}{2\widetilde{\lambda}-1+
2\sqrt{\left(\widetilde{\lambda}\right)^2-\widetilde{\lambda}}}-O(N^{-1})
\geq \frac{1}{4 \widetilde{\lambda}}-O(N^{-1}), \qquad N\geq 1,\; \widetilde{\lambda}\geq 1.
\end{equation}
On the other hand, we have
\begin{equation} \label{eq:baldda}
h_{N}^{\theta_{N}} \leq 
\frac{1}{\widetilde{\lambda}-2}+O(N^{-1}), \qquad\qquad \qquad N\geq 1,\;\widetilde{\lambda}\geq 3.
\end{equation}
Furthermore
\begin{equation} \label{eq:balduin}
h_{N}^{\theta_{N}} \leq 
\frac{2\widetilde{\lambda}}{2\widetilde{\lambda}-1
+\sqrt{4\widetilde{\lambda}^{2}-8 \widetilde{\lambda}+1}}+O(N^{-1}),
\qquad\qquad N\geq 1,\;\widetilde{\lambda}\geq 2,
\end{equation}
and for even $N$ moreover
\begin{equation} \label{eq:zusatz}
g_{N}^{\theta_{N}} \leq \frac{N+2}{N+4}, 
\qquad\qquad\qquad\qquad \widetilde{\lambda}=2,\; N\in\{2,4,6,\ldots\}.
\end{equation}
\end{theorem}

Observe that \eqref{eq:altegle} would only lead to a
lower bound of decay $O(N^{-1})$, instead of 
the absolute lower bound in \eqref{eq:neuy}.
Khintchine's transference principle \eqref{eq:khintchine}
combined with \eqref{eq:bern} admits no conclusion concerning
lower bounds for $h_{N}^{\theta_{N}}$ for large $N$,
and only yields an upper
bound of the form $1-O(N^{-1})$, 
reasonably weaker than \eqref{eq:baldda} for $\widetilde{\lambda}>2$,
and for 
$\widetilde{\lambda}=2$ slightly weaker than \eqref{eq:zusatz}. 
In both cases the implied 
constants depend on $\widetilde{\lambda}$ only.
For small parameters $\widetilde{\lambda}\in[1,\frac{3}{2})$, 
one readily checks that the bound \eqref{eq:neuy} is weaker 
than Beresnevich's bound \eqref{eq:beresn} as expected, unless
for $\widetilde{\lambda}=1$ when both equal $1$. The bound in
\eqref{eq:balduin} is stronger than \eqref{eq:baldda} for
parameters roughly in the interval
$\widetilde{\lambda}\in(2,3.5321\ldots)$. The bounds
in \eqref{eq:balduin} are larger than $\frac{1}{2}$
for any $\widetilde{\lambda}\geq 2$, whereas from \eqref{eq:beresn}
we expect $h_{N}^{\theta_{N}}=\frac{1}{2}$
for $\theta_{N}=3/(2N-1)$, which corresponds to
a parameter $\widetilde{\lambda}<\frac{3}{2}$ in the notation of
Theorem~\ref{dadthm}. However,
for $\widetilde{\lambda}>1.8+o(1)$ as $N\to\infty$, the bound
in \eqref{eq:baldda} is larger than the dimension formula 
in \eqref{eq:beresn} extended to the right. Thus for larger parameters \eqref{eq:beresn} can no longer represent the dimension formula 
for $h_{N}^{\lambda}$, 
as predicted for roughly $\widetilde{\lambda}\leq \frac{3}{2}$.

Now we investigate the case of fixed $\lambda>0$, and again aim
to derive asymptotic bounds for $h_{N}^{\lambda}$ as the
dimension $N$ grows.

\begin{theorem} \label{glanz}
Let $\lambda\in(0,1]$ be given. Then we have
\begin{equation} \label{eq:expre}
g_{N}^{\lambda}\leq  \frac{\lceil \lambda^{-1}\rceil+1}{N-2\lceil \lambda^{-1}\rceil+2}<\frac{\lambda^{-1}+2}{N}, 
\qquad\qquad \qquad N\geq 3\lceil \lambda^{-1}\rceil-1.
\end{equation}
Conversely, we have
\begin{equation}  \label{eq:stolt}
h_{N}^{\lambda} \geq 
\frac{(1+K)(1+\lambda-K\lambda)}{(N+1-K)(1+\lambda)},
\qquad\qquad K=\left\lfloor \frac{1+\lambda}{2\lambda}\right\rfloor, 
\quad  N\geq \left\lfloor\lambda^{-1}\right\rfloor.
\end{equation}
In particular, there exist 
positive constants $c_{1}(\lambda), c_{2}(\lambda)$ such that 
\[
\frac{c_{1}(\lambda)}{N} \leq h_{N}^{\lambda}\leq \frac{c_{2}(\lambda)}{N}, \qquad \qquad N\geq 1.
\]
\end{theorem}

The bound in \eqref{eq:expre} is good when $n$ is 
large compared to $\lambda^{-1}$.
For parameters $\lambda\in(\frac{1}{3},1]$, the bound
in \eqref{eq:stolt} coincides with \eqref{eq:altegle},
for $\lambda\leq \frac{1}{3}$ (hence $N\geq 3$) it provides
a strict improvement. It is discontinuous 
for $\lambda$ a reciprocal of an odd positive integer. 
For fixed $N$, a refined treatment 
leads to a slightly better continuous bound.
Moreover, similar to \eqref{eq:balduin} we can derive
effective piecewise constant upper bounds for fixed $N$.

\begin{theorem} \label{thmmet}
Let $N\geq 2$ be an integer. 
Define the intervals $I_{1}=[\frac{N+2}{3N},\infty)$ and
\[
I_{n}=\left[\frac{N+2}{N+2Nn+n-n^2},\frac{N+2}{N+2(n-1)N+(n-1)-(n-1)^2}\right), \quad 2\leq n\leq N.
\]
Then $I_{1},\ldots,I_{N}$ form a partition of $[\frac{1}{N},\infty)$
and we have
\begin{equation} \label{eq:tritt}
h_{N}^{\lambda} \geq \frac{(1+n)(1+\lambda-n\lambda)}{(1+\lambda)(N+1-n)},
\qquad\qquad \lambda\in{I_{n}}.
\end{equation}
The bound coincides with \eqref{eq:altegle} for $\lambda\in{I_{1}}$
and is strictly larger if 
otherwise $\lambda\in[\frac{1}{N},\frac{1}{3}+\frac{2}{3N})$.

Conversely, for $N/n\in(2,3)$ we have
\begin{equation} \label{eq:freitag}
g_{N}^{\lambda}\leq \frac{n+1}{N-n+1}, \qquad \text{if} \quad
\lambda\geq \frac{1}{(1-\gamma)N+(2\gamma-1)n},
\end{equation}
where 
\[
\gamma= \frac{N^{2}+2n^{2}-3Nn+2N-4n}{(n+1)(N-2n)}.
\]
\end{theorem}

Let $N\geq 4$. One checks that then we have
$\frac{1}{3}+\frac{2}{3N}>\frac{3}{2N-1}$. 
The bound of \eqref{eq:beresn} does not apply in the interval
$J_{N}=(\frac{3}{2N-1},\frac{1}{3}+\frac{2}{3N})$ and
hence \eqref{eq:tritt} provides the best known lower 
bound for $h_{N}^{\lambda}$ in $J_{N}$.

For $N=11$, the different bounds are illustrated by the Mathematica plots Figure~1 and Figure~2. 
The red curve depicting our new 
lower bounds is piecewise a rational function, 
which coincides with the green curve illustrating \eqref{eq:altegle} 
for $\lambda\geq 13/33=0.3939\ldots$, and exceeds 
it for smaller $\lambda$. Beresnevich's bound in blue decays almost linearly in its valid interval. Its continuation would exceed 
the red curve roughly up to $\lambda=0.1692\ldots$.
The piecewise constant gray line depicting upper bounds is
derived from \eqref{eq:freitag} with the choice $n=5$
in the interval $\lambda\in[0.2143,1/3]$, and from \eqref{eq:expre}
for $\lambda\geq 1/3$,
with discontinuities at $\lambda=0.2143\ldots$,
$\lambda=1/3$ and $\lambda=1/2$. We extended it to the 
interval $[1/11,0.2143\ldots]$ by the trivial bound $1$.

\begin{figure}
	\centering
\includegraphics[width=0.5\textwidth,natwidth=750,natheight=650]{specgraphik4.eps}
	\caption{Lower bounds for $h_{11}^{\lambda}$: Blue: Beresnevich's lower bound  \eqref{eq:beresn} in the valid interval $\lambda\in[\frac{1}{11},\frac{1}{7}]$. 
	Red: Our lower bound \eqref{eq:tritt} in the sample interval $\lambda\in[\frac{1}{11},0.225]$.
	Green: The lower bound \eqref{eq:altegle} in $\lambda\in[\frac{1}{11},0.225]$.}
\end{figure}

\begin{figure}[h!]
	\centering
	\includegraphics[width=0.5\textwidth,natwidth=750,natheight=650]{specgraphic_0323.eps}
	\caption{Bounds for $h_{11}^{\lambda}$: Blue, red, green as in Figure~1 in the interval $\lambda\in[\frac{1}{11},1]$. 
	Gray: Our upper bounds derived from \eqref{eq:expre}
	and \eqref{eq:freitag}.}
\end{figure}

Finally we remark that from Theorem~\ref{difizil} 
and \eqref{eq:bern} we may
derive very similar metric results on the dimensions of sets like
\[
r_{N}^{w}:= 
\dim(R_{N}^{w}), \qquad R_{N}^{w}:=\{ \zeta\in\mathbb{R}: \widehat{w}_{N}^{\ast}(\zeta)\leq w\}
\qquad \qquad N\geq 1, \; w\geq 1.
\]
We only state the particular consequence 
that for every fixed $\delta>0$ we have
$r_{N}^{N-\delta}\geq \frac{1}{4}-o(1)$ as $N\to\infty$.
Conversely if we fix $N$ and consider $r_{N}^{N-\delta}$
as $\delta\to 0$, the limit should be one 
provided that $w\mapsto r_{N}^{w}$ is continuous 
at $w=N$.

\section{Proofs} \label{profs}

\subsection{Deduction of the equivalence principles}
First we deduce Theorem~\ref{nabum} from the partial results
in Section~\ref{refin}.

\begin{proof}[Proof of Theorem~\ref{nabum}]
Theorem~\ref{unumber} shows that any $U_{m}$-number satisfies
\[
\lim_{n\to\infty} \lambda_{n}(\zeta)\geq \frac{1}{m-1}>0.
\]
On the other hand in Corollary~\ref{juchurra} we noticed that otherwise if 
$\zeta$ is no $U$-number, then $\lim_{n\to\infty} \lambda_{n}(\zeta)=0$.
Moreover, when $\zeta$ is a $U_{m}$-number, then $w_{m-1}(\zeta)<\infty$ and again Corollary~\ref{juchurra} yields that
we actually have $\lambda_{n}(\zeta)\leq \frac{1}{m-1}$ for large $n$, 
so by the above observation there must be equality. 
In Theorem~\ref{tnumber} we proved that for $T$-numbers we have 
$\limsup_{n\to\infty} n\lambda_{n}(\zeta)=\infty$, and
$\lim_{n\to\infty}\lambda_{n}(\zeta)=0$ was shown above.
Finally the claim for $S$-numbers was noticed in 
Corollary~\ref{juchurra} as well.
\end{proof}

We now settle the second equivalence principle Theorem~\ref{dadkor}
and Theorem~\ref{difizil}.
Lower bounds for $\overline{\widehat{w}}^{\ast}(\zeta)$
are based on Theorem~\ref{nabum} and 
the relations
\begin{equation} \label{eq:besides}
\widehat{w}_{n}^{\ast}(\zeta)\geq \frac{1}{\lambda_{n}(\zeta)}, \qquad \qquad w_{n}^{\ast}(\zeta)\geq \frac{1}{\widehat{\lambda}_{n}(\zeta)},
\end{equation}
see~\cite{davsh} and~\cite{j2}. For upper bounds we
employ recent results from~\cite{bschlei}.

\begin{proof}[Proof of Theorem~\ref{dadkor}]
By \eqref{eq:besides} and Theorem~\ref{nabum}, for \eqref{eq:limites}
to hold, $\zeta$ must be a $U$-number. 
The refined result on $U_{m}$-numbers
in Theorem~\ref{nabum} moreover implies that for $\zeta$ a $U_{m}$-number we have
$\widehat{w}_{n}^{\ast}(\zeta)\geq m-1$ for large $n$. On the other hand,
it was shown in~\cite[Corollary~2.5]{bschlei} that for any $U_{m}$-number we have $\widehat{w}_{n}^{\ast}(\zeta)\leq m$ for all $n\geq 1$.
The above implies the left property of \eqref{eq:tnu} for
$S$ and $T$-numbers. We next prove the right claim in \eqref{eq:tnu}
for $T$-numbers. It was shown
in~\cite[Theorem~2.4]{bschlei} that for $m,n$ positive integers
the estimate $w_{m}(\zeta)>m+n-1$ implies
\begin{equation} \label{eq:traene}
\widehat{w}_{n}^{\ast}(\zeta)\leq 
m+(n-1)\frac{\widehat{w}_{n}^{\ast}(\zeta)}{w_{m}(\zeta)}.
\end{equation}
For a $T$-number $\zeta$ and every integer $N$ we have
$w_{m}(\zeta)\geq N^{2}m$ for some $m$. If we choose $n=Nm$, then 
the condition $w_{m}(\zeta)>m+n-1$ is satisfied when $N\geq 2$.
From \eqref{eq:bound2} and \eqref{eq:traene} we infer
\[
\widehat{w}_{mN}^{\ast}(\zeta)\leq m+\frac{2(mN)^{2}}{N^{2}m}
\leq 3m.
\]
Hence indeed $\widehat{w}_{n}^{\ast}(\zeta)/n=\widehat{w}_{mN}^{\ast}(\zeta)/(mN)\leq 3/N$ which tends to $0$
as $N\to\infty$. Finally, for $S$-numbers we derive
$\widehat{w}_{n}^{\ast}(\zeta)\gg n$ from \eqref{eq:besides} and 
Theorem~\ref{nabum}, the converse follows from above.
\end{proof}

\begin{remark}
Besides \eqref{eq:besides}, 
the inequalities
\begin{equation} \label{eq:najo}
\widehat{w}_{n}^{\ast}(\zeta)\geq \frac{w_{n}(\zeta)}{w_{n}(\zeta)-n+1}, \qquad w_{n}^{\ast}(\zeta)\geq 
\frac{\widehat{w}_{n}(\zeta)}{\widehat{w}_{n}(\zeta)-n+1}
\end{equation}
linking $w_{n}^{\ast}(\zeta)$ and $\widehat{w}_{n}^{\ast}(\zeta)$ 
with other classical exponents due to
Bugeaud and Laurent~\cite{buglaur} are known.
However, \eqref{eq:najo}
implies $\lim_{n\to\infty} \widehat{w}_{n}^{\ast}(\zeta)=\infty$
only upon the considerably stronger
condition $\liminf_{n\to\infty} w_{n}(\zeta)/n=1$. Similarly,
a uniform lower bound for the quantities 
$\widehat{w}_{n}^{\ast}(\zeta)/n$ would require a uniform upper bound on 
$w_{n}(\zeta)-n$ instead of $w_{n}(\zeta)/n$.
\end{remark}

We remark that we can obtain the variants of Theorem~\ref{dadkor}
mentioned in Remark~\ref{reeh}
by considering the corresponding variants of \eqref{eq:besides}.
Again relation \eqref{eq:besides} and a refined treatment of the argument 
for $T$-numbers leads to a proof of Theorem~\ref{difizil}.

\begin{proof}[Proof of Theorem~\ref{difizil}]
The respective left inequalities in \eqref{eq:hum} 
and \eqref{eq:bug} and $\theta(\zeta)\geq \tau(\zeta)^{-1}$ 
follow immediately from Theorem~\ref{genau}
and \eqref{eq:besides}. 
Concerning the respective right inequalities, the estimates 
$\overline{\widehat{w}}^{\ast}(\zeta)\leq \overline{\widehat{w}}(\zeta)
\leq \overline{w}(\zeta)$ and $\underline{\widehat{w}}^{\ast}(\zeta)\leq \underline{\widehat{w}}(\zeta)
\leq \underline{w}(\zeta)$
are an easy consequence of \eqref{eq:wmono} and \eqref{eq:bound2}. 
For the remaining bounds, we refine the argument 
for $T$-numbers in the proof 
of Theorem~\ref{dadkor}. We may assume 
$\overline{w}(\zeta)>2$ and $\underline{w}(\zeta)>2$ respectively, 
otherwise the left bounds are smaller and the claim is obvious.
So assume $\alpha>2$ is a fixed real number and 
$m$ is a large integer such that
$w_{m}(\zeta)/m>\alpha$. If $n$ is another integer
and we define $\beta= n/m$, then 
in the case of $\beta\leq \alpha-1$ the condition 
$w_{m}(\zeta)>m+n-1$ of
\eqref{eq:traene}  is satisfied. Its application and rearrangements 
yield
\[
\widehat{w}_{n}^{\ast}(\zeta)\leq 
\frac{\alpha}{\alpha-\beta} m.
\]
Dividing by $n=\beta m$ yields
\[
\frac{\widehat{w}_{n}^{\ast}(\zeta)}{n}\leq 
\frac{\alpha}{(\alpha-\beta)\beta}.
\]
Let $n=\lfloor m\alpha/2\rfloor$. Then
$\beta=n/m=\alpha/2+O(1/m)$. Hence,
since for $\alpha>2$ we have $\alpha/2<\alpha-1$, 
the above condition $\beta\leq \alpha-1$ is satisfied for large $m$.
By inserting we obtain the upper bound
$4/\alpha+O(1/m)$ for $\widehat{w}_{n}^{\ast}(\zeta)/n$.
By definition we may choose $\alpha$ arbitrarily close to 
$\overline{w}$ for certain arbitrarily large $m$, and 
to $\underline{w}$ for all large $m$, respectively. The claims 
\eqref{eq:hum} 
and \eqref{eq:bug} follow. Finally we show $\theta(\zeta)\leq \tau(\zeta)^{-1}$ to settle \eqref{eq:hadschibratschi}.
Let $\epsilon>0$ and assume $w_{m}(\zeta)\geq m^{\gamma}$
for some $\gamma>1$. 
Let $n=m^{\gamma-\epsilon}$ 
and observe that again the condition $w_{m}(\zeta)>m+n-1$ is satisfied
for large $m$. Thus by \eqref{eq:traene}, again for 
large enough $m\geq m_{0}(\epsilon)$, we infer
\[
\widehat{w}_{n}^{\ast}(\zeta)\leq 
\frac{m^{\gamma+1}}{m^{\gamma}-m^{\gamma-\epsilon}+1}
\leq 2m= 2 n^{1/(\gamma-\epsilon)}.
\]
Hence taking logarithms to base $n$ gives $\theta(\zeta)\leq \tau(\zeta)^{-1}$
as $\gamma$ can be chosen arbitrarily close to $\tau(\zeta)$ 
for certain arbitrarily large $m$ and
$\epsilon$ arbitrarily small.
\end{proof}

We place the proof of Theorem~\ref{genau} in Section~\ref{kummtscho}
as it requires some partial results of the proof of Theorem~\ref{tnumber}.

\subsection{Proofs of the upper bounds}

Next we show Theorem~\ref{juchu}. The proof is
similar to~\cite[Theorem~2.1]{ichrelations}. Recall the
successive minima exponents defined in Section~\ref{a}.
As noticed in~\cite{j2}, Mahler's duality can be formulated in the way
\begin{equation} \label{eq:mahla}
\lambda_{n,j}(\zeta)= \frac{1}{\widehat{w}_{n,n+2-j}(\zeta)}, \qquad \widehat{\lambda}_{n,j}(\zeta)= \frac{1}{w_{n,n+2-j}(\zeta)},
\end{equation}
with the successive minima exponents defined below 
Theorem~\ref{unumber}. Indeed, our proof for the
upper bounds for $\lambda_{N}(\zeta)$
are based on controlling the uniform last successive minimum
exponent of the dual problem $\widehat{w}_{N,N+1}(\zeta)$, for 
suitable $N$. In fact the dual concept is underlying
any proof of upper bounds for $\widehat{\lambda}_{n}(\zeta)$,
where control of $w_{n,n+1}(\zeta)$ leads to
bounds for $\widehat{\lambda}_{n}(\zeta)$.
We also recall Gelfond's Lemma, asserting that
\begin{equation} \label{eq:wirsing}
H(P)H(Q) \ll_{n} H(PQ) \ll_{n} H(P)H(Q)
\end{equation}
holds for any polynomials $P,Q$ each of degree at most $n$.

\begin{proof}[Proof of Theorem~\ref{juchu}]
Let $n,\zeta$ be as in the theorem and $\epsilon>0$. 
By definition of $\widehat{w}_{n}(\zeta)$, for any 
large $X\geq X_{0}(\epsilon)$ there
exists an integer polynomial $P_{X}$ of degree at most $n$ such that
\[
H(P_{X})\leq X, \qquad \vert P_{X}(\zeta)\vert \leq X^{-\widehat{w}_{n}(\zeta)+\epsilon}.
\]
Now choose an integer $k\geq w_{n}(\zeta)$. The definition 
of $\widehat{w}_{k}(\zeta)$ similarly yields
an integer polynomial $Q_{X}$ of degree at most $k$ such that 
\begin{equation} \label{eq:hoert}
H(Q_{X})\leq X, \qquad \vert Q_{X}(\zeta)\vert 
\leq X^{-\widehat{w}_{k}(\zeta)+\epsilon}.
\end{equation}
Write $Q_{X}= R_{X}S_{X}$, where $R_{X}$ consists of the factors dividing $P_{X}$ as well, 
and $S_{X}$ is coprime to $P_{X}$. Let $\epsilon>0$.
We claim that, unless $R_{X}$ is of small height $H(R_{X})\ll_{\epsilon} 1$, we have
\begin{equation} \label{eq:hoerthoert}
\vert R_{X}(\zeta)\vert \geq H(R_{X})^{-w_{n}(\zeta)-\epsilon}\gg_{k,\zeta} X^{-w_{n}(\zeta)-\epsilon},
\end{equation}
if $X$ was chosen sufficiently large. First notice that the corresponding estimate
\begin{equation} \label{eq:obtain}
\vert U_{X}(\zeta)\vert \geq H(U_{X})^{-w_{n}(\zeta)-\epsilon},
\end{equation}
applies to any irreducible factor $U_{X}$ of $R_{X}$. 
Indeed, such $U_{X}$ has degree at most $n$ as it also divides $P_{X}$, 
and by definition of $w_{n}(\zeta)$ we obtain \eqref{eq:obtain}.
From \eqref{eq:wirsing} we see that this property is essentially (up to a factor depending on $k$ only)
preserved when taking arbitrary products, 
more precisely
\begin{equation} \label{eq:hoerthoerti}
\vert R_{X}(\zeta)\vert \geq H(R_{X})^{-w_{n}(\zeta)-\epsilon}\gg_{k,\zeta} X^{-w_{n}(\zeta)-\epsilon}.
\end{equation}
%
In case of $R_{X}$ of small height $H(R_{X})\ll_{\epsilon} 1$, we can even estimate 
$\vert R_{X}(\zeta)\vert\gg_{n,\zeta} 1$ by the finiteness and 
since $\zeta$ is transcendental. 
From \eqref{eq:hoert} and \eqref{eq:hoerthoert} we deduce
\[
\vert S_{X}(\zeta)\vert = \frac{\vert Q_{X}(\zeta)\vert }{\vert R_{X}(\zeta)\vert } \leq X^{-\widehat{w}_{k}(\zeta)+w_{n}(\zeta)+2\epsilon}.
\]
Moreover, since $S_{X}$ divides $Q_{X}$, Gelfond's estimate \eqref{eq:wirsing} implies $H(S_{X})\ll_{k} H(Q_{X})\leq X$. Hence we have
\begin{equation} \label{eq:inview}
\max\{H(P_{X}), H(S_{X})\}\ll_{k} X, \qquad  \max\{ \vert P_{X}(\zeta)\vert, \vert S_{X}(\zeta)\vert\} \leq X^{-\theta_{k,n}+2\epsilon},
\end{equation}
with
\[
\theta_{k,n}= \min \{ \widehat{w}_{n}(\zeta), \widehat{w}_{k}(\zeta)-w_{n}(\zeta) \}.
\]
Let $d_{X}=d\leq n$ be the degree of $P_{X}$ and $e_{X}=e\leq k$ be the degree of $S_{X}$.
Then, since $P_{X}$ and $Q_{X}$ are coprime, the set of polynomials
\[
\mathscr{P}_{X}:= \{ P_{X}, TP_{X},\ldots, T^{e-1}P_{X}, S_{X},TS_{X}, \ldots, T^{d-1}S_{X}\}
\]
is linearly independent and spans the space of polynomials of 
degree at most $d+e-1\leq k+n-1$.
In case of strict inequality
$d+e-1<k+n-1$ for some $X$, we consider
\[
\mathscr{R}_{X}= \mathscr{P}_{X}\cup \{T^{d}S_{X},T^{d+1}S_{X},\ldots,T^{k+n-1-e}S_{X}\}
\]
instead of $\mathscr{P}_{X}$ (see also the proof 
of Proposition~\ref{wpropos} below).
Clearly $\mathscr{R}_{X}$ is linearly independent as well, 
and spans the space of polynomial of degree at most $N:=k+n-1$.
In any case,
in view of \eqref{eq:inview} and since $X$ was arbitrary and we may choose $\epsilon$ arbitrarily small,
this means
\[
\widehat{w}_{N,N+1}(\zeta)\geq \theta_{k,n}
=\min\{ \widehat{w}_{n}(\zeta), \widehat{w}_{N-n+1}(\zeta)-w_{n}(\zeta) \}.
\]
Since $\theta_{k,n}>0$ by construction,
Mahler's relation \eqref{eq:mahla} with $j=1$ further implies $\lambda_{N}(\zeta)\leq 1/\theta_{k,n}$. We may choose 
any integer $k> w_{n}(\zeta)$, and the choice
$k=\lceil w_{n}(\zeta)\rceil$ yields
$N=n+k-1=\lceil w_{n}(\zeta)\rceil+n-1$.
The claim \eqref{eq:vor1} follows. 

Now we prove \eqref{eq:vor4}. We now choose an integer
$k$ with strict 
inequality $k>w_{n}(\zeta)$, and again obtain \eqref{eq:hoert}
for some $Q_{X}$ of degree at most $k$ for any $X\geq X_{0}(\epsilon)$. 
We proceed as above splitting
$Q_{X}=R_{X}S_{X}$. By a very similar argument as above,
from \eqref{eq:wirsing}     
we derive that $Q_{X}$ cannot split solely in
irreducible polynomials of degree at most $n$. Thus it
must have an irreducible factor
of degree at least $n+1$, which must divide $S_{X}$. 
Hence $R_{X}= Q_{X}/S_{X}$ has degree
at most $k-(n+1)$. In particular if $w_{n}(\zeta)<n+1$, for $k=n+1$
we infer $S_{X}= Q_{X}$ and $R_{X}\equiv 1$ for all large $X$.
From the definition of $w_{k-n-1}$ for sufficiently large $H(R_{X})$ 
we derive
\begin{equation} \label{eq:hoerthoertii}
\vert R_{X}(\zeta)\vert \geq H(R_{X})^{-w_{k-n-1}(\zeta)-\epsilon}\gg_{k,\zeta} X^{-w_{k-n-1}(\zeta)-\epsilon}.
\end{equation}
In case of small heights of $R_{X}$ we use the argument 
from the proof of \eqref{eq:vor1} again.
We infer \eqref{eq:inview} very similarly as above
with $\theta_{k,n}$ replaced by the new expression
\[
\tilde{\theta}_{k,n}= \min \{ \widehat{w}_{n}(\zeta), 
\widehat{w}_{k}(\zeta)-w_{k-n-1}(\zeta) \}.
\]
Let $N=k+n-1$ again, proceeding as above yields
\[
\widehat{w}_{N,N+1}(\zeta)\geq \tilde{\theta}_{k,n}
=\min\{ \widehat{w}_{n}(\zeta), \widehat{w}_{N-n+1}(\zeta)-w_{N-2n}(\zeta) \}.
\]
We may start with
any integer $k> w_{n}(\zeta)$, or equivalently
$k\geq \lfloor w_{n}(\zeta)\rfloor+1$, which leads 
to $N\geq \lfloor w_{n}(\zeta)\rfloor+n$.
The claim \eqref{eq:vor4} follows from \eqref{eq:mahla} again
as soon as $\tilde{\theta}_{k,n}>0$, which we can
guarantee for $N\leq 3n$ by construction. The condition
$w_{n}<2n-1$ is only required for the set of values $N$ 
in \eqref{eq:vor4} to be non-empty.
\end{proof}

\subsection{Parametric geometry of numbers}

The proofs of Section~\ref{a} and Section~\ref{spectrum} 
can be derived in a convenient, and in fact surprisingly easy way,
utilizing the parametric geometry
of numbers introduced by Schmidt and Summerer~\cite{ss}. We 
recall the fundamental concepts, in a slightly modified form to
fit our purposes. In particular 
we restrict to successive powers of a number.
Let $\zeta\in{\mathbb{R}}$ be given and
$Q>1$ a parameter. For $n\geq 1$ and $1\leq j\leq n+1$
define $\psi_{n,j}(Q)$ as the minimum of $\eta\in{\mathbb{R}}$ such that
\begin{equation} \label{eq:droben}
\vert x\vert \leq Q^{1+\eta}, \qquad \max_{1\leq j\leq n} \vert \zeta^{j}x-y_{j}\vert\leq Q^{-\frac{1}{n}+\eta} 
\end{equation}
has $j$ linearly independent solution vectors $(x,y_{1},\ldots,y_{n})\in{\mathbb{Z}^{n+1}}$. 
The functions $\psi_{n,j}(Q)$ can be equivalently defined via a lattice point problem, see~\cite{ss}. As pointed out in~\cite{ss}
they have the properties
\begin{equation}  \label{eq:heut}
-1\leq \psi_{n,1}(Q)\leq \psi_{n,2}(Q)\leq \cdots 
\leq \psi_{n,n+1}(Q)\leq \frac{1}{n}, \qquad\qquad Q>1.
\end{equation}
Let 
\begin{equation}  \label{eq:fest1}
\underline{\psi}_{n,j}=\liminf_{Q\to\infty} \psi_{n,j}(Q),
\qquad \overline{\psi}_{n,j}=\limsup_{Q\to\infty} \psi_{n,j}(Q).
\end{equation}
These values all belong to the interval $[-1,1/n]$ by \eqref{eq:heut}. 
From Dirichlet's Theorem it follows that
$\psi_{n,1}(Q)<0$ for all $Q>1$ and hence $\underline{\psi}_{n,1}\leq 0$.
Similarly, for $1\leq j\leq n+1$ define the 
functions $\psi_{n,j}^{\ast}(Q)$ as the infimum of $\eta$ 
such that
\begin{equation}  \label{eq:char}
H(P)\leq Q^{\frac{1}{n}+\eta}, \qquad \vert P(\zeta)\vert \leq Q^{-1+\eta}
\end{equation}
have $j$ linearly independent integer polynomial solutions
$P$ of degree at most $n$.
%
%
Again put
\begin{equation} \label{eq:fest2}
\underline{\psi}_{n,j}^{\ast}=\liminf_{Q\to\infty} \psi_{n,j}^{\ast}(Q),
\qquad \overline{\psi}_{n,j}^{\ast}=\limsup_{Q\to\infty} \psi_{n,j}^{\ast}(Q).
\end{equation}
We have
\[
-\frac{1}{n}\leq \psi_{n,1}^{\ast}(Q)\leq \psi_{n,2}^{\ast}(Q)\leq 
\cdots \leq \psi_{n,n+1}^{\ast}(Q)\leq
1, \qquad Q>1.
\]
As pointed out in~\cite{ss} Mahler's relations 
\eqref{eq:mahla} are essentially equivalent to
\begin{equation} \label{eq:jaja}
\underline{\psi}_{n,j}=-\overline{\psi}_{n,n+2-j}^{\ast}, \qquad 
\overline{\psi}_{n,j}=-\underline{\psi}_{n,n+2-j}^{\ast}, \qquad 1\leq j\leq n+1.
\end{equation}
Schmidt and Summerer~\cite[(1.11)]{ssch} further established the inequalities 
\begin{equation} \label{eq:simul}
j\underline{\psi}_{n,j}+(n+1-j)\overline{\psi}_{n,n+1}\geq 0, 
\qquad j\overline{\psi}_{n,j}+(n+1-j)\underline{\psi}_{n,n+1}\geq 0,
\end{equation}
for $1\leq j\leq n+1$. The dual inequalities
\begin{equation} \label{eq:duales}
j\underline{\psi}_{n,j}^{\ast}+(n+1-j)\overline{\psi}_{n,n+1}^{\ast}\geq 0, 
\qquad j\overline{\psi}_{n,j}^{\ast}+(n+1-j)\underline{\psi}_{n,n+1}^{\ast}\geq 0,
\end{equation}
can be obtained very similarly. We point out that
in the case of equality in \eqref{eq:simul}
and \eqref{eq:duales} respectively,
their proofs in~\cite{ss} directly imply
\begin{equation} \label{eq:vieh}
\underline{\psi}_{n,1}=\underline{\psi}_{n,2}=\cdots=\underline{\psi}_{n,j}, \qquad\qquad
\overline{\psi}_{n,j+1}=\overline{\psi}_{n,j+2}=\cdots
=\overline{\psi}_{n,n+1},
\end{equation}
and
\begin{equation} \label{eq:varvieh}
\underline{\psi}_{n,1}^{\ast}=\underline{\psi}_{n,2}^{\ast}=\cdots=\underline{\psi}_{n,j}^{\ast}, \qquad\qquad
\overline{\psi}_{n,j+1}^{\ast}=\overline{\psi}_{n,j+2}^{\ast}=\cdots=\overline{\psi}_{n,n+1}^{\ast},
\end{equation}
respectively.
Moreover, by~\cite[Theorem~1.4]{ss} the quantities 
in \eqref{eq:fest1} and \eqref{eq:fest2}
are connected to the 
exponents $\lambda_{n,j},\widehat{\lambda}_{n,j}$ and $w_{n,j},\widehat{w}_{n,j}$ via the identities
\begin{equation} \label{eq:umrechnen}
(1+\lambda_{n,j}(\zeta))(1+\underline{\psi}_{n,j})=
(1+\widehat{\lambda}_{n,j}(\zeta))(1+\overline{\psi}_{n,j})=\frac{n+1}{n}, \qquad 1\leq j\leq n+1,
\end{equation}
and 
\begin{equation} \label{eq:umrechnen2}
(1+w_{n,j}(\zeta))\Big(\frac{1}{n}+\underline{\psi}_{n,j}^{\ast}\Big)=
(1+\widehat{w}_{n,j}(\zeta))\Big(\frac{1}{n}+\overline{\psi}_{n,j}^{\ast}\Big)=\frac{n+1}{n}, \qquad 1\leq j\leq n+1.
\end{equation}
In fact it was only observed for $j=1$ in~\cite{ss}, 
but as remarked in~\cite{j2} it is true as well
for $2\leq j\leq n+1$ for the same reason. 

\subsection{Proofs of the lower bounds} 

The following easy observation will play an important role in the
proofs of lower bounds.

\begin{proposition} \label{wpropos}
Let $m,n$ be positive integers and $\zeta$ be a real transcendental number. Then 
\begin{equation} \label{eq:schnee}
w_{m+n,m+i}(\zeta)\geq w_{n,i}(\zeta), \qquad \widehat{w}_{m+n,m+i}(\zeta)\geq \widehat{w}_{n,i}(\zeta),
\qquad 1\leq i\leq n+1.
\end{equation}
\end{proposition}

\begin{proof}
By the definition of $w_{n,i}(\zeta)$, for certain arbitrarily large $X$
there exist linearly independent integer polynomials
$P_{1},\ldots, P_{i}$ of degree at most $n$ with the properties
\[
\max_{1\leq j\leq i} H(P_{j}) \leq X, \qquad 
\max_{1\leq j\leq i} \vert P_{j}(\zeta)\vert 
\leq H(P_{j})^{-w_{n,i}(\zeta)+\epsilon}.
\]
Without loss of generality assume the degree of $P_{1}$ is maximal
among the $P_{j}$. For any $m\geq 1$ consider the set
of polynomials
\[
\mathcal{P}_{m,n,i}=\mathcal{P}_{m,n,i}(X)=
\{P_{1}(T),TP_{1}(T),T^{2}P_{1}(T),\ldots ,T^{m}P_{1}(T),P_{2}(T),\ldots,P_{i}(T)\}
\]
It is not hard to see that $\mathcal{P}_{m,n,i}$ consists of 
$m+i$ polynomials
of degree at most $n+m$,
which are linearly independent as well, and satisfies 
\[
\max_{P\in{\mathcal{P}_{m,n,i}}} H(P) \leq X, \qquad 
\max_{P\in{\mathcal{P}_{m,n,i}}} \vert P(\zeta)\vert 
\leq \max\{1,\vert \zeta\vert^{m}\} H(P)^{-w_{n,i}(\zeta)+\epsilon}.
\]
The left inequalities of \eqref{eq:schnee} follow. The right ones
are shown similarly using the definition of $\widehat{w}_{n,i}(\zeta)$
and considering any large $X$.
\end{proof}

In fact we only need the case $i=1$. First
we deduce Theorem~\ref{unumber} from the proposition.

\begin{proof} [Proof of Theorem~\ref{unumber}]
By \eqref{eq:reihenfolge} it suffices to prove \eqref{eq:neuheit} for $n\geq m$. So let $n=m+k$ with $k\geq 0$.
From $w_{m}(\zeta)=\infty$ and Proposition~\ref{wpropos} we derive $w_{n,k+1}(\zeta)=\infty$. 
Together with \eqref{eq:umrechnen2} we infer
\[
\underline{\psi}_{n,k+1}^{\ast}=-\frac{1}{n}.
\]
Hence \eqref{eq:duales} with $j=k+1$ and \eqref{eq:jaja} yield
\begin{equation} \label{eq:vdb}
\underline{\psi}_{n,1}=-\overline{\psi}_{n,n+1}^{\ast}\leq \frac{k+1}{(n+1)-(k+1)}\underline{\psi}_{n,k+1}^{\ast}
= -\frac{k+1}{(k+m)m}.
\end{equation}
Inserting in \eqref{eq:umrechnen} yields $\lambda_{n}(\zeta)\geq 1/(m-1)$ as asserted.
Now assume $m\geq 2$ and $w_{m-1}(\zeta)=m-1$. We first only show $\lambda_{n}(\zeta)=1$ for $n\geq m-1$.
The properties $w_{m-1}(\zeta)=m-1$ and $\lambda_{m-1}(\zeta)=1/(m-1)$ are equivalent,
which follows for example from \eqref{eq:khintchine}. 
Thus for $n\geq m-1$ we have $1/(m-1)=\lambda_{m-1}\geq \lambda_{n}\geq 1/(m-1)$ by
\eqref{eq:neuheit} and \eqref{eq:reihenfolge}, and our special case of \eqref{eq:neueun} follows.
For the general claim \eqref{eq:neueun}, observe that reversing the above process,
the identity $\lambda_{n}(\zeta)=1/m$ implies equality in the inequality in \eqref{eq:vdb}. 
As observed in \eqref{eq:vieh} and \eqref{eq:varvieh} above,
this can only happen when $\overline{\psi}_{n,k+2}^{\ast}=\overline{\psi}_{n,k+3}^{\ast}=\cdots=\overline{\psi}_{n,n+1}^{\ast}$.
Noticing $n+1-(k+1)=m$,
Mahler's relations \eqref{eq:jaja} yield
$\underline{\psi}_{n,1}=\underline{\psi}_{n,2}=\cdots=\underline{\psi}_{n,m}$,
and by \eqref{eq:umrechnen2} we infer $\lambda_{n,1}(\zeta)=\cdots=\lambda_{n,m}(\zeta)$.
The last claim follows similarly from Khintchine's principle \eqref{eq:khintchine}.
\end{proof}

Similar considerations and a well-known existence result
lead to a proof of Theorem~\ref{besser}.

\begin{proof}[Proof of Theorem~\ref{besser}]
From combining Theorem~\ref{unumber} and \eqref{eq:trampel} we deduce \eqref{eq:1} and \eqref{eq:2}, apart 
from $\lambda_{n,2}(\zeta)=1$.
This refinement is inferred similarly as the general claim of \eqref{eq:neueun} 
in the proof of Theorem~\ref{unumber}.
Indeed, since
we already know $\lambda_{n}(\zeta)=1$, the same argument
implies $\overline{\psi}_{n,n}^{\ast}=\overline{\psi}_{n,n+1}^{\ast}$
or equivalently $\underline{\psi}_{n,1}=\underline{\psi}_{n,2}$,
and \eqref{eq:umrechnen} yields $\lambda_{n,2}(\zeta)=\lambda_{n}(\zeta)=1$. In order to show that any sequence as in \eqref{eq:3} 
belongs to the joint spectrum,
it suffices to notice that $U_{2}$-numbers with any prescribed value $w=w_{1}(\zeta)$ can be constructed by means
of continued fractions as pointed out in~\cite[paragraph~7.6 on page~158]{bugbuch}. Eventually, from Theorem~\ref{nabum}
(or Theorem~\ref{juchu}) we know that $\lim_{n\to\infty} \lambda_{n}(\zeta)>1/2$ implies $\zeta$ must be a 
$U_{1}$-number or $U_{2}$-number. In case of a $U_{2}$-number the above applies, for a Liouville number $\zeta$ the 
sequence $(\lambda_{n}(\zeta))_{n\geq 1}$
takes the value $\lambda_{n}(\zeta)=\infty$ anyway for all $n\geq 1$ as noticed in~\cite[Corollary~2]{bug}. 
\end{proof}

Theorem~\ref{tnumber} follows similarly as Theorem~\ref{unumber}, with slightly more computation involved. 

\begin{proof}[Proof of Theorem~\ref{tnumber}]
Let $m,n$ be positive integers and $C\geq 1$ a real number to be chosen later and assume
we have $w_{n}(\zeta)\geq Cn$. Proposition~\ref{wpropos}
yields
\[
w_{m+n,m+1}(\zeta)\geq nC.
\]
With \eqref{eq:umrechnen2} we obtain
\[
\underline{\psi}_{m+n,m+1}^{\ast}\leq \frac{m+n+1}{(m+n)(1+nC)}-\frac{1}{m+n}=\frac{m+n(1-C)}{(m+n)(1+nC)}.
\]
Hence \eqref{eq:duales} with $j=m+1$ and \eqref{eq:jaja} imply
\begin{equation} \label{eq:hofa}
\underline{\psi}_{m+n,1}=-\overline{\psi}_{m+n,m+n+1}^{\ast}\leq \frac{m+1}{n}\underline{\psi}_{m+n,m+1}^{\ast}
= \frac{m+1}{n}\cdot\frac{m+n(1-C)}{(m+n)(1+nC)}.
\end{equation}
Application of \eqref{eq:umrechnen} yields
\begin{equation} \label{eq:bratschi}
\lambda_{m+n}(\zeta)\geq \frac{m+n+1}{m+n}\cdot\frac{1}{1+\frac{m+1}{n}\frac{m+n(1-C)}{(m+n)(1+nC)}} -1=\frac{Cn-m}{m+n(1+C(n-1))}.
\end{equation}
Let $m=\lceil Rn\rceil$ with the optimal parameter $R=(C-1)/2$. A short computation shows
\begin{equation} \label{eq:dtrump}
(m+n)\lambda_{m+n}(\zeta)\geq \left(\frac{C+1}{2}\right)^{2}\frac{n}{R+1+(n-1)C}-\epsilon>\frac{(C+1)^{2}}{4C}-2\epsilon,
\end{equation}
for $n\geq n_{0}(C,\epsilon)$. We infer \eqref{eq:wiedefwoben} as we may choose $C$ arbitrarily close to $\overline{w}$ and $\underline{w}$
respectively, for certain arbitrarily large $n$ and all large $n$,
respectively.
For $T$-numbers we may choose arbitrarily large $C$ 
for certain large $n$, and the claim \eqref{eq:altheit} follows.
\end{proof}

We kind of dualize the proof for the uniform exponents.

\begin{proof}[Proof of Theorem~\ref{uniformiert}]
We first modify the proof of Theorem~\ref{tnumber} to
show the left inequalities of \eqref{eq:family}.
We apply the uniform inequality of Proposition~\ref{wpropos}
to see that if $\widehat{w}_{n}(\zeta)\geq Cn$ then for any $m\geq 1$ 
we have
\[
\widehat{w}_{m+n,m+1}(\zeta)\geq nC.
\]
With \eqref{eq:umrechnen2} we infer
\[
\overline{\psi}_{m+n,m+1}^{\ast}\leq \frac{m+n+1}{(m+n)(1+nC)}-\frac{1}{m+n}=\frac{m+n(1-C)}{(m+n)(1+nC)}.
\]
Again \eqref{eq:duales} with $j=m+1$ and \eqref{eq:jaja} yield
\[
\overline{\psi}_{m+n,1}=-\underline{\psi}_{m+n,m+n+1}^{\ast}\leq \frac{m+1}{n}\overline{\psi}_{m+n,m+1}^{\ast}
= \frac{m+1}{n}\cdot\frac{m+n(1-C)}{(m+n)(1+nC)}.
\]
We apply \eqref{eq:umrechnen} and obtain
\[
\widehat{\lambda}_{m+n}(\zeta)\geq \frac{Cn-m}{m+n(1+C(n-1))}.
\]
Again with the parameter choice $m=\lceil n(C-1)/2\rceil$ we
obtain the left inequalities in \eqref{eq:family} by multiplication 
with $m+n$. The right inequalities of \eqref{eq:family} are a
consequence of \eqref{eq:ribiza} as we show now.
Let $\epsilon>0$.
By definition of $\overline{\widehat{w}}(\zeta)$ there exist arbitrarily
large $n$ such that $\widehat{w}_{n}(\zeta)
\geq n(\overline{\widehat{w}}(\zeta)-\epsilon)$. 
Provided $n$ was chosen sufficiently large,
for
$m=\lfloor n(\overline{\widehat{w}}(\zeta)-\epsilon)\rfloor$,
from \eqref{eq:ribiza} and $w_{m}(\zeta)\geq m$ we infer
\[
\widehat{\lambda}_{m+n-1}(\zeta)=
\widehat{\lambda}_{\lfloor n(\overline{\widehat{w}}(\zeta)-\epsilon)\rfloor+n-1}(\zeta)\leq
\max\left\{ \frac{1}{m}, \frac{1}{\widehat{w}_{n}(\zeta)}\right\}\leq 
\frac{1}{n(\overline{\widehat{w}}(\zeta)-2\epsilon)}.
\]
Since $m+n-1\leq n(1+\overline{\widehat{w}}(\zeta)-\epsilon)$
we obtain
\[
(m+n-1)\widehat{\lambda}_{m+n-1}(\zeta)\leq 
\frac{1+\overline{\widehat{w}}(\zeta)-\epsilon}{\overline{\widehat{w}}(\zeta)-2\epsilon}.
\]
We conclude
$\underline{\widehat{\lambda}}(\zeta)\leq 
(\overline{\widehat{w}}(\zeta)+1)/\overline{\widehat{w}}(\zeta)$ 
as $\epsilon\to 0$ and $n\to\infty$. 
The claim 
$\overline{\widehat{\lambda}}(\zeta)\leq 
(\widehat{\underline{w}}(\zeta)+1)/\widehat{\underline{w}}(\zeta)$ 
is derived similarly, starting with
any large $n$. Finally \eqref{eq:day} follows from 
\eqref{eq:bound2}, \eqref{eq:vergleich}
and Theorem~\ref{difizil}.
\end{proof}

\subsection{Proofs of Theorem~\ref{bugadam} and Theorem~\ref{genau}}
\label{kummtscho}
We now deduce Theorem~\ref{bugadam} from Theorem~\ref{tnumber}.

\begin{proof}[Proof of Theorem~\ref{bugadam}]
First we show that \eqref{eq:uaah} implies \eqref{eq:tja}.
Let $N= \lceil w_{n}(\zeta)+\widehat{w}_{n}(\zeta)\rceil+n-1$
and apply \eqref{eq:vor1}.
We check that the left
bound is larger and hence
\begin{equation} \label{eq:priskard}
\lambda_{N}(\zeta)\leq \frac{1}{\widehat{w}_{n}(\zeta)}.
\end{equation}
On the other hand,
when we write $N=m+(n+u)$, such that 
$m=\lceil w_{n}(\zeta)+\widehat{w}_{n}(\zeta)\rceil-(u+1)$, 
from \eqref{eq:bratschi} and 
$N\leq w_{n}(\zeta)+\widehat{w}_{n}(\zeta)+n$ we obtain
\begin{equation} \label{eq:turmrd}
\lambda_{N}(\zeta) \geq \frac{w_{n+u}(\zeta)-m}
{N+(n+u-1)w_{n+u}(\zeta)}
\geq \frac{w_{n+u}(\zeta)-m}
{w_{n}(\zeta)+\widehat{w}_{n}(\zeta)+(n+u-1)w_{n+u}(\zeta)+n}.
\end{equation}
Since $\widehat{w}_{n}(\zeta)-(n+u-1)>0$ by assumption,
combination of \eqref{eq:priskard} and \eqref{eq:turmrd} 
yields after elementary rearrangements
\begin{equation} \label{eq:diejohannard}
w_{n+u}(\zeta) \leq 
\frac{(m+1)\widehat{w}_{n}(\zeta)+w_{n}(\zeta)+n}{\widehat{w}_{n}(\zeta)-(n+u-1)}.
\end{equation}
By assumption \eqref{eq:uaah} we may apply \eqref{eq:salsa},
and for the index $j=0$ we obtain
\begin{equation}  \label{eq:luciar}
w_{n}(\zeta)\leq 
\frac{(n-1)\widehat{w}_{n}(\zeta)}{\widehat{w}_{n}(\zeta)-n}.
\end{equation}
Plugging this in \eqref{eq:diejohannard} and
estimating $m\leq  w_{n}(\zeta)+\widehat{w}_{n}(\zeta)-u$, 
after a short calculation we obtain \eqref{eq:tja}.

For the unconditional bound, notice that the right
hand side in \eqref{eq:diejohannard} is monotonic decreasing
as a function of $\widehat{w}_{n}(\zeta)$ in the interval
$(n+u-1,\infty)$. Thus, by $\widehat{w}_{n}^{\ast}(\zeta)\leq \widehat{w}_{n}(\zeta)$ from \eqref{eq:bound2}, upon the assumption 
$\widehat{w}_{n}^{\ast}(\zeta)-(n+u-1)>0$ from \eqref{eq:diejohannard}
we may conclude
\begin{equation} \label{eq:diejohanna}
w_{n+u}(\zeta) \leq 
\frac{(m+1)\widehat{w}_{n}^{\ast}(\zeta)+w_{n}(\zeta)+n}{\widehat{w}_{n}^{\ast}(\zeta)-(n+u-1)}.
\end{equation}
Again since $\widehat{w}_{n}^{\ast}(\zeta)-n\geq \widehat{w}_{n}^{\ast}(\zeta)-(n+u-1)>0$, we apply the unconditional estimate 
\eqref{eq:sturz} for $j=0$ which reads
\begin{equation}  \label{eq:luciarde}
w_{n}(\zeta)\leq 
\frac{(n-1)\widehat{w}_{n}^{\ast}(\zeta)}{\widehat{w}_{n}^{\ast}(\zeta)-n},
\end{equation}
and the same calculation 
as above shows \eqref{eq:unkon}.
Finally, the asymptotic claim \eqref{eq:asyms}
is obtained from \eqref{eq:bound2}.
\end{proof}

Finally we put the results together to prove Theorem~\ref{genau}. 

\begin{proof}[Proof of Theorem~\ref{genau}]
The inequalities in \eqref{eq:leftie}
have already been noticed in Theorem~\ref{tnumber} and Theorem~\ref{juchu}. 
The inequality $\sigma(\zeta)\leq -1/\tau(\zeta)$ follows easily from \eqref{eq:vor2}.
The reverse inequality $\sigma(\zeta)\geq -1/\tau(\zeta)$ follows from \eqref{eq:bratschi}
by taking $m=n^{\tau(\zeta)-\delta}$ and letting $\delta\to 0$, 
observing that $\log C/\log n$ is arbitrarily close to $\tau(\zeta)-1$ for certain large $n$.
\end{proof}

\subsection{Proofs of the metric results}
For the proof of Theorem~\ref{dadthm} we essentially reverse
the proofs of Theorem~\ref{juchu} and Theorem~\ref{tnumber}.
We have to be careful with the occurring error terms
from rounding to integers in the process. For simplicity let
\[
t_{n}^{w}= \dim(T_{n}^{w}), \qquad\qquad
T_{n}^{w}= \{ \zeta\in\mathbb{R}: w_{n}(\zeta)\geq w\}.
\]
By \eqref{eq:bern} we have $t_{n}^{w}= (n+1)/(w+1)$ 
for $n\geq 1, w\geq n$.

\begin{proof}[Proof of Theorem~\ref{dadthm}]
For $\widetilde{\lambda}>1$ as in the theorem let
\[
C= 2\widetilde{\lambda}-1+
2\sqrt{\left(\widetilde{\lambda}\right)^2-\widetilde{\lambda}},
\qquad \sigma= \frac{2}{1+C},
\]
where $C \geq 1$ is a solution
to $(C+1)^{2}/(4C)=\widetilde{\lambda}$.
Assume $N$ is large and 
further let $n= \left\lceil n^{\ast}\right\rceil$ with
$n^{\ast}= \sigma N$, such that
\[
n= \left\lceil n^{\ast}\right\rceil
=\left\lceil \sigma N\right\rceil
= \left\lceil\frac{2}{1+C}N\right\rceil.
\]
Obviously $0\leq \sigma N-n<1$.
Assume $\zeta$ satisfies
$w_{n}(\zeta)\geq Cn$.
We want to show that $\lambda_{N}(\zeta)$ is
then essentially bounded below by $\theta_{N}$. 
We proceed as in Theorem~\ref{tnumber} for our present $n$.
For any integer $m\geq 1$, we obtain 
\begin{equation} \label{eq:obben}
(m+n)\lambda_{m+n}(\zeta)\geq 
\frac{(m+n)(C n-m)}{m+n(1+C(n-1))}.
\end{equation}
We choose $m=\left\lceil m^{\ast}\right\rceil$ with
$m^{\ast}= N-n^{\ast}=n^{\ast}\cdot (C-1)/2$.
Since
\begin{equation}  \label{eq:fani}
0\leq m-m^{\ast}<1, \qquad 0\leq n-n^{\ast}<1,
\end{equation}
we conclude
\begin{equation} \label{eq:neux}
m+n\in \{N,N+1\}.  
\end{equation}
Denote by $\Phi(m,n)$ the right hand side in \eqref{eq:obben} 
treated as a function in two variables.
When we replace $m,n$ by $m^{\ast}, n^{\ast}$, a computation shows
\[
\Phi(m^{\ast},n^{\ast})= 
\frac{ n^{\ast}(\frac{C+1}{2})^2}{C(n^{\ast}-\frac{1}{2})+\frac{1}{2}-C}
\]
We readily deduce that for some constant $c_{0}= c_{0}(C)$ we have
\[
\Phi(m^{\ast},n^{\ast})\geq  
\frac{(C+1)^{2}}{4C}-\frac{c_{0}}{n}.
\]
Similarly with \eqref{eq:fani} we easily verify that for some 
constant $c_{1}=c_{1}(C)$ we have 
\[
\Phi(m,n)\geq \Phi(m^{\ast},n^{\ast})-\frac{c_{1}}{n}.
\]
Since $n\gg_{C} N$ by construction, 
combination of the two inequalities yields
\[
\Phi(m,n)\geq  
\frac{(C+1)^{2}}{4C}-\frac{c_{2}}{n}
\geq \frac{(C+1)^{2}}{4C}-\frac{c_{3}}{N},
\]
with $c_{2}(C)= c_{0}(C)+c_{1}(C)$
and some new constant $c_{3}=c_{3}(C)$.
Since $\Phi(m,n)$ was the lower bound in \eqref{eq:obben}
and by \eqref{eq:neux}, for some new constant $c_{4}= c_{4}(C)$
we infer
\[
N \lambda_{N}(\zeta) \geq \frac{N}{N+1}\Phi(m,n) \geq \frac{N}{N+1}\cdot 
\left(\frac{(C+1)^{2}}{4C}-\frac{c_{3}}{N}\right)
\geq  \frac{(C+1)^{2}}{4C}-\frac{c_{4}}{N}.
\]
Since $(C+1)^{2}/(4C)=\widetilde{\lambda}$, the argument shows
\[
T_{n}^{Cn}
\subseteq
H_{N}^{\frac{\varphi}{N}}, \qquad\qquad
\varphi=\widetilde{\lambda}-\frac{c_{4}}{N}.
\]
Thus by \eqref{eq:bern} we estimate the dimension as
\[
h_{N}^{\frac{\varphi}{N}}\geq 
t_{n}^{Cn}=
\frac{n+1}{Cn+1}\geq \frac{1}{C}.
\]
The claim \eqref{eq:neuy} follows by starting with
$\widetilde{\lambda}+\epsilon$ instead of $\widetilde{\lambda}$
and letting $\epsilon$ tend to $0$.
The upper bound \eqref{eq:baldda} follows similarly
from~\eqref{eq:juchur}, the proof is left to the reader.

We show \eqref{eq:balduin}. 
Let $\widetilde{\lambda}\geq 2$ be given
and assume $N\lambda_{N}(\zeta)>\widetilde{\lambda}$ for some large $N$. 
Let $\beta\in[2,3), \gamma\geq 1$ to be chosen later. For now 
for simplicity assume $n=N/\beta$ is an integer. 
Assume $w_{n}(\zeta)\leq (\beta-1) n$. Then since $\beta\in[2,3)$
we may apply \eqref{eq:vor4} to $N,n$ and obtain
\begin{equation} \label{eq:pjo}
\lambda_{m}(\zeta)\leq 
\frac{1}{m-n+1-w_{m-2n}(\zeta)}
\leq \frac{1}{(\beta-1) n-w_{(\beta-2)n}(\zeta)}
\end{equation}
Now assume
\[
w_{(\beta-2)n}(\zeta)\leq \gamma (\beta-2)n.
\]
Then multiplication of \eqref{eq:pjo} with $N=\beta n$ yields
\begin{equation}   \label{eq:wirdscho}
N\lambda_{N}(\zeta)\leq \frac{\beta}{\beta-1-\gamma(\beta-2)}.
\end{equation}
In other words, if we have 
\[
N\lambda_{N}(\zeta)>\frac{\beta}{\beta-1-\gamma(\beta-2)},
\]
then either $w_{n}(\zeta)>(\beta-1) n$ or $w_{(\beta-2)n}(\zeta)> \gamma (\beta-2)n$. Hence
\begin{equation} \label{eq:hoerer}
\mathcal{D}:=\mathcal{D}_{N,\beta,\gamma}:=G_{N}^{\frac{\beta}{N(\beta-1-\gamma(\beta-2))}}=
\{ \zeta\in\mathbb{R}: N\lambda_{N}(\zeta)>\frac{\beta}{\beta-1-\gamma(\beta-2)}\}
\end{equation}
is contained in 
$\mathcal{A}\cup \mathcal{B}=\mathcal{A}_{N,\beta}\cup \mathcal{B}_{N,\beta,\gamma}$ with
\begin{equation}  \label{eq:aandb}
\mathcal{A}:= T_{n}^{(\beta-1)n}, \qquad
\mathcal{B}:= T_{(\beta-2)n}^{\gamma (\beta-2)n}.
\end{equation}
The dimensions of the sets $\mathcal{A},\mathcal{B}$
can be determined with \eqref{eq:bern} as 
\[
\dim(\mathcal{A})
= \frac{n+1}{(\beta-1) n+1}= \frac{1}{\beta-1}+O(n^{-1}),
\quad
\dim(\mathcal{B})
= \frac{(\beta-2)n+1}{\gamma(\beta-2)n+1}= \frac{1}{\gamma}+O(n^{-1}).
\]
Hence, for given $\beta\in[2,3), \gamma\geq 1$, the 
set $\mathcal{D}$ in \eqref{eq:hoerer} has Hausdorff dimension at most 
\begin{equation}  \label{eq:dbound}
\dim(\mathcal{D})\leq \dim(\mathcal{A}\cup \mathcal{B})= 
\max \{ \dim(\mathcal{A}), \dim(\mathcal{B})\}=
\max\{ \frac{1}{\beta-1}, \frac{1}{\gamma}\}+O(n^{-1}).
\end{equation}
Put
\begin{equation}  \label{eq:saturnalien}
\gamma= \frac{2\widetilde{\lambda}-1+\sqrt{4\widetilde{\lambda}^2-8\widetilde{\lambda}+1}}{2\widetilde{\lambda}}, \qquad 
\beta=\gamma+1.
\end{equation}
For these choices one checks
that the right hand side expression in \eqref{eq:hoerer} equals
$\widetilde{\lambda}$. On the other hand, by construction clearly 
$(\beta-1)^{-1}=\gamma^{-1}$ and this value yields
the right hand side of \eqref{eq:balduin} as an upper bound in \eqref{eq:dbound}. 
One further readily checks
that for any given $\widetilde{\lambda}\geq 2$, the initial assumptions
$\beta\in[2,3)$ and $\gamma\geq 1$ for $\beta, \gamma$
in \eqref{eq:saturnalien} are satisfied.
We still have to deal with the problem that
$N/\beta$ is no integer in general.
However, for given $\beta\in[2,3)$ if we
let $n=\lfloor N/\beta\rfloor$ then
$\beta^{\prime}:=N/n=\beta+O(N^{-1})$. 
One checks that the above procedure starting with $\beta^{\prime}$ 
instead of $\beta$ leads to an additional error 
at most $O(N^{-1})$. 
We have proved \eqref{eq:balduin}.

Finally we show \eqref{eq:zusatz}. Within the proof 
of Theorem~\ref{unifthm},
for even $N$ we have established the inclusion
\[
G_{N}^{\frac{2}{N}} \subseteq
T_{\frac{N}{2}}^{\frac{N}{2}+1}.
\]
Thus $g_{N}^{2/N}\leq t_{N/2}^{N/2+1}= (N/2+1)/(N/2+2)=(N+2)/(N+4)$ 
follows from \eqref{eq:bern} again.
\end{proof}

Theorem~\ref{glanz} is obtained in a similar way.

\begin{proof}[Proof of Theorem~\ref{glanz}]
Assume $\lambda>0$ and $\zeta$ satisfies 
$\lambda_{N}(\zeta)>\lambda$. 
If we let $k=\lceil \lambda^{-1}\rceil$ then~\eqref{eq:vor2} implies
$w_{k}(\zeta)>N-2k+1$. In other 
words $G_{N}^{\lambda} \subseteq T_{k}^{N-2k+1}$.
Hence, when $N-2k+1\geq k$ or equivalently $N\geq 3k-1$, the formula \eqref{eq:bern} implies \eqref{eq:expre}. 
The left lower bound in \eqref{eq:stolt} 
rephrases \eqref{eq:altegle}.
To prove the right bound, we again reverse
the proof of Theorem~\ref{juchu}. In
the estimate \eqref{eq:bratschi} identify $N=m+n$. 
We see that if for $n\in\{1,2,\ldots,N\}$ and $C\geq 1$
the identity
\[
\frac{C n-(N-n)}{N-n+n(1+C (n-1))}= \lambda
\]
holds, then $T_{n}^{Cn} \subseteq H_{N}^{\lambda}$.
By \eqref{eq:bern}, we further conclude
\begin{equation} \label{eq:tadaaa}
h_{N}^{\lambda} \geq t_{n}^{Cn}= \frac{n+1}{nC+1}.
\end{equation}
In other words, we have
\begin{equation}  \label{eq:pride}
h_{N}^{\lambda} \geq \frac{n+1}{nC+1}, \qquad\qquad \text{if} \quad
C= \frac{N(\lambda+1)-n}{(\lambda+1)n-\lambda n^2}. 
\end{equation}
Inserting the expression for $C$, after some simplification we derive
\begin{equation} \label{eq:schauher}
h_{N}^{\lambda} \geq \max_{1\leq n\leq N}
\frac{(1+n)(1+\lambda-n\lambda)}{(1+\lambda)(N+1-n)}.
\end{equation}
We may choose
\[
n= \left\lfloor \frac{\lambda^{-1}+1}{2} \right\rfloor=
\left\lfloor \frac{\lambda+1}{2\lambda} \right\rfloor,
\]
since $N\geq \lambda^{-1}\geq (1+\lambda)/(2\lambda)\geq n\geq 1$ follows from the assumption $\lambda\leq 1$. Insertion 
in \eqref{eq:pride} yields \eqref{eq:stolt}.
\end{proof}

We infer~\eqref{eq:tritt} with a refined treatment of 
\eqref{eq:schauher}, and \eqref{eq:freitag} similarly to \eqref{eq:balduin}. We only sketch the proof.

\begin{proof}[Proof of Theorem~\ref{thmmet}]
It can easily be checked that the right hand side maximum in \eqref{eq:schauher} is obtained for the integer
$n$ if $\lambda\in{I_{n}}$,
and yields the bounds \eqref{eq:tritt} in the theorem.
The estimate \eqref{eq:freitag} can be derived
with method of the proof of \eqref{eq:balduin}
and using the precise dimension formula for $\mathcal{A}$ 
and $\mathcal{B}$ defined in \eqref{eq:aandb}, 
we skip the computations.
\end{proof}

\end{document}